\documentclass[onefignum,onetabnum]{siamart220329}
\usepackage{graphics,graphicx,epstopdf,url}
\usepackage{arydshln}
\usepackage{color}
\usepackage{algorithmic}
\usepackage{url,cases,supertabular}
\usepackage{amssymb,mathtools}
\usepackage{enumerate,makecell}
\usepackage{amsfonts}
\usepackage{graphicx,epstopdf}
\usepackage{color,verbatim}
\usepackage{mathrsfs,subeqnarray,subfigure}
\usepackage{listings,array}
\usepackage{booktabs}
\usepackage{dashrule}
\usepackage{arydshln}
\usepackage{graphicx}
\usepackage{amsfonts}
\usepackage{mathrsfs}
\usepackage{graphicx}  \usepackage{epstopdf}
\usepackage{subfigure}  \usepackage{lineno}
\usepackage{multirow,bm}
\usepackage{color}
\usepackage{algorithm}
\usepackage{algorithmic}
\usepackage{longtable}
\usepackage{booktabs}
\usepackage{rotating}
\usepackage{enumitem}
\usepackage{lineno}
\usepackage{url}
\usepackage{pdflscape}
\usepackage{makecell}
\usepackage{tikz}
\usepackage{tcolorbox}

\numberwithin{equation}{section}

\newtheorem{assumption}[theorem]{Assumption}
 \newtheorem{remark}[theorem]{Remark}

\newcommand{\ba}{\begin{array}}
\newcommand{\ea}{\end{array}}

\newcommand{\bit}{\begin{itemize}}
\newcommand{\eit}{\end{itemize}}
\newcommand{\be}{\begin{equation}}
\newcommand{\ee}{\end{equation}}
\newcommand{\bea}{\begin{eqnarray}}
\newcommand{\eea}{\end{eqnarray}}

\newcommand{\st}{\mathrm{s.t.}}

\newcommand{\argmin}{\mathop{\mathrm{arg\,min}}}

\newcommand{\by}{\mathrm{\bf y}}

\newcommand{\bx}{\mathbf{x}}
\newcommand{\bW}{\mathbf{W}}

\newcommand{\Rmn}[1]{\uppercase\expandafter{\romannumeral#1}}
\renewcommand{\theenumi}{\roman{enumi}}

\newcommand{\Pcal}{\mathcal{P}}

\numberwithin{equation}{section}

\newcommand{\Mcal}{\mathcal{M}}

\newcommand{\grad}{\mathrm{grad}}

\newcommand{\R}{\mathbb{R}}

\numberwithin{theorem}{section}
\newcommand{\iprod}[2]{\left \langle #1, #2 \right \rangle }

\usepackage{lipsum}
\usepackage{clrscode}
\usepackage{appendix}
\usepackage{bm}
\usepackage{booktabs}
\usepackage{url}
\usepackage{multirow}
\usepackage{textcomp}
\usepackage{amsmath}
\usepackage{amsfonts}
\usepackage{amssymb}
\usepackage{mathrsfs}
\usepackage{hyperref}
\usepackage{graphicx,graphics,subfigure}
\usepackage{hyperref,url}
\usepackage{epsf,epstopdf}
\usepackage{algorithm}
\usepackage{algorithmic}
\usepackage{bbm}
\usepackage{dsfont}
\usepackage{indentfirst}
\usepackage{pdfpages}
\usepackage{pdflscape}
\usepackage{longtable}
\ifpdf
  \DeclareGraphicsExtensions{.eps,.pdf,.png,.jpg}
\else
  \DeclareGraphicsExtensions{.eps}
\fi

\headers{Achieving Consensus over Compact Submanifolds}{Jiang Hu, JiaoJiao Zhang, and Kangkang Deng}

\title{Achieving consensus over compact submanifolds}

\author{Jiang Hu\thanks{Massachusetts General Hospital and Harvard Medical School, Harvard University, Boston, MA 02114
(\email{hujiangopt@gmail.com}).}
    \and Jiaojiao Zhang\thanks{Division of Decision and Control Systems, KTH Royal Institute of Technology Stockholm, Sweden (\email{jiaoz@kth.se}).}
  \and  Kangkang Deng\thanks{Corresponding author. Department of Mathematics,  National University of Defense Technology, Changsha, 410073,
China (\email{freedeng1208@gmail.com}).}
}

\usepackage{amsopn}

\ifpdf
\hypersetup{
  pdftitle={Achieving consensus over compact submanifolds },
  pdfauthor={}
}
\fi

\begin{document}

\maketitle

\begin{abstract}
We consider the consensus problem in a decentralized network, focusing on a compact submanifold that acts as a nonconvex constraint set. By leveraging the proximal smoothness of the compact submanifold, which encompasses  the local singleton property and the local Lipschitz continuity of the projection operator on the manifold, and establishing the connection between the projection operator and general retraction, we show that the Riemannian gradient descent with a unit step size has locally linear convergence if the network has a satisfactory level of connectivity. 
Moreover, based on the geometry of the compact submanifold, we prove that a convexity-like regularity condition, referred to as the restricted secant inequality, always holds in an explicitly characterized neighborhood around the solution set of the  nonconvex consensus problem. By leveraging this restricted secant inequality and imposing a weaker connectivity requirement on the decentralized network, we present a comprehensive analysis of the linear convergence of the Riemannian gradient descent, taking into consideration appropriate initialization and step size. 
Furthermore, if the network is well connected, we demonstrate that the local Lipschitz continuity endowed by proximal smoothness is a sufficient condition for the restricted secant inequality, thus contributing to the local error bound. We believe that our established results will find more application in the consensus problems over a more general proximally smooth set. Numerical experiments are conducted to validate our theoretical findings.

\end{abstract}


\section{Introduction}
In a decentralized system, a group of agents collaborates to minimize a global loss function through neighborhood communication. Consensus serves as a fundamental aspect of decentralized optimization, aiming to make all agents agree on a common state.
Consider an undirected graph denoted as $\mathcal{G} = [\mathcal{V}, \mathcal{E}]$, where $\mathcal{V}$ represents the set of agents in the graph and $\mathcal{E}$ represents the set of edges. We define two agents as neighbors if they are connected by an edge. The problem of decentralized consensus strives to achieve a state of agreement among all nodes through local computation and communication with neighbors. To facilitate this communication among nodes, we introduce a mixing matrix $W$, where $W_{ij} > 0$ indicates the existence of an edge between nodes $i$ and $j$. We assume that $W_{ii} > 0$ holds true, implying that agent $i$ can always communicate with itself.

Consensus over a manifold has gained substantial attention in the optimization and control community over the past decade \cite{markdahl2020high, sarlette2009consensus, tron2012riemannian}. Mathematically, it can be formulated as follows:
\be \label{prob:original}
\begin{aligned} 
  \min_{\bx} \quad & \varphi^t(\bx):= \frac{1}{4}\sum_{i=1}^N \sum_{j=1}^N W_{ij}^t \|x_i - x_j\|^2, \\
  \st \quad & x_i \in \Mcal, \;\; \forall i=1,2,\ldots, N,
\end{aligned}  
\ee 
where $N$ is the number of agents, $\bx = [x_1^\top, x_2^\top, \ldots, x_N^\top]^\top \in \R^{(Nd) \times r}$, $W \in \R^{N\times N}$ is the mixing matrix, $t>0$ is an integer, $W^t$ is the $t$-th power of $W$, $W_{ij}^t$ is the $(i,j)$-th element of $W^t$, and $\Mcal$ is a submanifold of $\R^{d\times r}$. 

The consensus over a manifold has garnered significant attention as a crucial element in decentralized manifold optimization methods, such as  decentralized principal component analysis \cite{shah2017distributed,ye2021deepca,chen2021decentralized,gang2022linearly}, decentralized low-rank matrix completion \cite{mishra2019riemannian,hu2023decentralized,deng2023decentralized}, and decentralized low-dimension subspace learning \cite{mishra2019riemannian,hu2023decentralized,deng2023decentralized}.
Besides, it finds direct applications in diverse areas, including the sensor network \cite{tron2009distributed,paley2009stabilization}, the Lohe model of quantum synchronization \cite{ha2018relaxation,markdahl2018geometric}, and the Kuramoto models \cite{rodrigues2016kuramoto,markdahl2020high}.

\subsection{Literature review}
In the Euclidean setting (i.e., $\Mcal = \R^{d\times r}$), the problem \eqref{prob:original} is convex,  and the gradient descent method exhibits a globally linear convergence rate when the second-largest singular value of $W^t$ is strictly less than $1$ \cite{nedic2018network}. Such convexity and globally linear convergence of consensus algorithms  play a crucial role in the design of various decentralized optimization algorithms \cite{tsitsiklis1986distributed,nedic2009distributed,shi2015extra,yuan2016convergence,xu2015augmented,qu2017harnessing,scutari2019distributed}.

In the case of submanifold (i.e., $\Mcal \ne \R^{d\times r}$  ), there are two formulations for the consensus problem. One follows the formulation in \eqref{prob:original}, while the other is based on the geodesic distance on the manifold, replacing $\|x_i - x_j \|^2$ with ${\rm dist}^2(x_i,x_j)$, where ${\rm dist}$ represents the geodesic distance \cite{tron2012riemannian}.
According to \cite{tron2012riemannian}, these two formulations are referred to as the extrinsic and intrinsic approaches, respectively. By leveraging the geometric properties of the manifold, the Riemannian gradient descent \cite{absil2009optimization,hu2020brief,boumal2023introduction} is a popular method for solving the consensus problem. However, for the intrinsic consensus problem, the computational cost is high due to the involvement of both exponential and logarithmic mappings on the manifold in the Riemannian gradient update \cite{tron2012riemannian}. Although the Riemannian gradient descent can achieve consensus where all agents converge to the same point, it is unclear whether it guarantees a linear convergence rate.
On the other hand, the extrinsic consensus is initially considered in \cite{sarlette2009consensus,markdahl2020high} for Stiefel-like manifolds, such as the special orthogonal group and the Grassmann manifold. In comparison to the intrinsic consensus, the computational cost of performing one-step Riemannian gradient descent for the extrinsic consensus problem is significantly reduced as there is no need to calculate the logarithmic mapping. By utilizing the proximal smoothness of the compact submanifold, the projected gradient descent with a unit step size has been proven to converge linearly \cite{deng2023decentralized} in a neighborhood of optimal solution. However, it is not yet clear whether the Riemannian gradient descent exhibits a (local) linear convergence rate in this context.

For the case of $\Mcal$ being the Stiefel manifold, the work \cite{chen2021local} shows that Riemannian gradient descent converges linearly when solving the extrinsic consensus \eqref{prob:original} with proper initialization and suitable step size. They establish the restricted secant inequality in a neighborhood around the global optima by digging into the geometry of the Stiefel manifold and the doubly stochastic property of $W$. It should be noticed that their analysis relies on the specific structure of the Stiefel manifold and is not directly applicable to general submanifolds. Besides, it remains unclear whether the unit step size is acceptable. 

\subsection{Contribution}
We study the regularity conditions around the global optima and the locally linear convergence of Riemannian gradient descent for solving \eqref{prob:original}. The comparisons of our results with the existing works are summarized in Table \ref{tab:comparison}. Specifically, our contributions are as follows: 
\begin{itemize}
    \item With the local Lipschitz continuity of the projection on $\Mcal$ endowed by the proximal smoothness of the compact submanifold and the connection between projection and general retractions in Lemma \ref{lem:lip-proj}, we show that the Riemannian gradient descent with a unit step size has locally linear convergence. In particular, we give the explicit characterization of the local neighborhood and show all the iterates stay in the neighborhood under certain conditions, which enables us to use the local Lipschitz continuity to obtain the convergence. 
    
    \item We establish in Lemma \ref{lem:conseus-bound} the isometric property between the Euclidean consensus error and the manifold consensus error, which are defined by the distances to the Euclidean mean and the induced arithmetic mean, respectively. By leveraging a combination of the Lipschitz-type inequalities of the projection operator and the retraction operator, the normal inequality, and the spectrum of $W$, we show the validity of the restricted secant inequality within a precisely characterized neighborhood for problem \eqref{prob:original} with any $t \geq 1$ in Theorem \ref{thm:rsi}. By such restricted secant inequality, we then present the linear convergence analysis of Riemannian gradient descent with appropriate initialization and step size.

    \item We also show that the restricted secant inequality of \eqref{prob:original} can be derived directly from the local Lipschitz continuity of the projection on $\Mcal$ when $t$ is large enough. Although the restricted secant inequality of \eqref{prob:original} holds for all $t \geq 1$, such a result provides a more streamlined analysis of the restricted secant property. Additionally, the local error bound of \eqref{prob:original} holds naturally as it is weaker than the restricted secant condition.
    The relations between these regularity conditions are summarized in Figure \ref{fig:compare-regular}, and we anticipate broader applications of this result in the context of consensus over a more general proximally smooth set.
\end{itemize}

\begin{table}[t] 
	\caption{Comparison of methods for solving consensus problem \eqref{prob:original}. Here, RGD and PGD represent Riemannian gradient descent and projected gradient descent, respectively. The ``Convexity" column indicates whether the corresponding consensus problem is convex.  The ``Unit Step Size" column assesses whether the unit step size is acceptable for the linear convergence of the associated method.} \vspace{0.1cm}
	\centering
        \setlength{\tabcolsep}{1mm}{} 
	\begin{tabular}{|p{3cm}|c|c|c|c|}
        \cline{1-5}
         & Constraint set & Method  & Convexity & Unit step size \\
        \cline{1-5}
        Nedi\'{c}, Ozdaglar, and Parrilo \cite{nedic2010constrained}  & convex set & PGD & convex & yes \\  
        \cline{1-5}
        Chen et al. \cite{chen2021local} & Stiefel manifold & RGD & nonconvex & no \\
        \cline{1-5}
        Deng and Hu \cite{deng2023decentralized} & Compact submanifold & PGD & nonconvex & yes \\ 
        \cline{1-5}
        \textbf{This work} & Compact submanifold & RGD & nonconvex & yes \\ 
        \cline{1-5}		
\end{tabular}
\label{tab:comparison}
\end{table}

\usetikzlibrary{matrix}
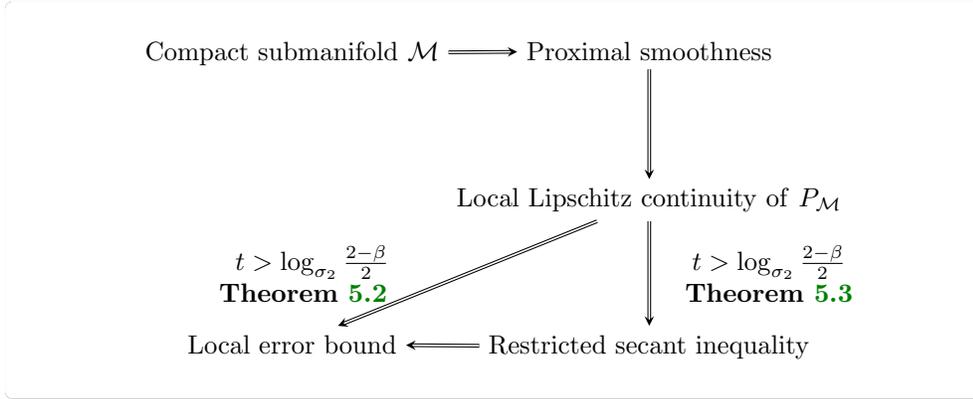
\begin{figure}[t]
\begin{center}
	\begin{tcolorbox}[colback = white, boxrule = {0.01pt}]
            \centering
		\begin{tikzpicture}
			\matrix (m) [matrix of math nodes,row sep=4em,column sep=0em,minimum width=0.5em]
			{
                {\rm Compact~submanifold~}\Mcal  & {\rm Proximal~ smoothness}  \\ 
                & {\rm Local~Lipschitz~continuity~of~} P_{\mathcal{M}} \\
			{\rm Local~error~bound} &  {\rm Restricted~secant~inequality}  \\};
			\path[-stealth]
			(m-3-2) edge [double] node [right] {} (m-3-1)
			(m-2-2) edge [double] node [left, align=center] {$t > \log_{\sigma_2} \frac{2-\beta}{2} \qquad$ \\ \textbf{Theorem \ref{lem:eb2} \qquad}
   } (m-3-1)
			(m-2-2)	edge [double] node [right, align=center] {$\quad t > \log_{\sigma_2} \frac{2-\beta}{2}$ \\
   \quad \textbf{Theorem \ref{thm:Lip-rsi}}} (m-3-2)
                (m-1-1)	edge [double] node [right] {} (m-1-2)
                (m-1-2)	edge [double] node [right] {} (m-2-2);
		\end{tikzpicture}
	\end{tcolorbox}
\end{center} 
\caption{Implications of different concepts of regularity conditions in terms of problem \eqref{prob:original}. Here, $\beta \in (0,2)$ is a given constant and $\sigma_2 < 1$ is the second largest singular value of $W$.}
\label{fig:compare-regular}
\end{figure}

\subsection{Notation} For a positive integer $N$, we denote $[N]=\{1,2,\ldots, N\}$ and $J= \frac{1}{N} \mathbf{1}_N\mathbf{1}_N^\top$ with $\mathbf{1}_N \in \R^N$ being a vector of all entries equal to $1$. For a matrix $x \in \R^{d\times r}$, we denote its Euclidean norm as $\|x\|$. Let $\bW^t= W^t \otimes I_d \in \R^{(nd) \times (nd)}$ for a positive integer $t$, where $\otimes$ is the Kronecker product. For the submanifold $\Mcal \subset \R^{d\times r}$, we always set the Euclidean metric as the Riemannian metric. We denote the tangent space  and the normal space of $\Mcal$ at a point $x$ as $T_x\Mcal$ and $N_x\Mcal$, respectively. For a differentiable function $f:\R^{d\times r} \rightarrow \R$, we denote its Euclidean gradient and Riemannian gradient as $\nabla f(x)$ and $\grad f(x)$, respectively. We denote the $n$-fold Cartesian product of $\Mcal$ as $\Mcal^N = \underbrace{\Mcal \times \cdots \times \Mcal}_{N}$.

\section{Preliminary}
\subsection{Manifold optimization}
Manifold optimization has attracted much attention in the past few decades, as evident in works such as   \cite{absil2009optimization,hu2020brief,boumal2023introduction}. The goal of manifold optimization is to minimize a real-valued function over a manifold, i.e.,
\[ \min_{x \in \Mcal} \quad f(x), \]
where $\Mcal$ is a Riemannian manifold and $f:\Mcal \rightarrow \R$ is a real-valued function. If $\Mcal$ is a submanifold embedded in $\R^{d\times r}$ and the function $f$ can be extended to $\R^{d\times r}$, then the Riemannian gradient of $f$ at $x$ can be computed as $\grad f(x) = P_{T_{x}\Mcal}(\nabla f(x))$, where $P_{T_x\Mcal}$ represents the orthogonal projection onto $T_x \Mcal$. In the design of Riemannian algorithms, an essential concept  is the so-called retraction operator.  A retraction operator $R$ at $x$, denoted as $R_x$, is a mapping from $T_x\Mcal$ to $\Mcal$ that satisfies the following two conditions:
\begin{itemize}
    \item $R_x(0_x) = x$, where $0_x$ is the zero element of $T_x \Mcal$.
    \item $\frac{\rm d}{{\rm d}t}R_{x}(t\xi_x)\mid_{t=0} = \xi_x$ for any $\xi_x \in T_x \Mcal$. 
\end{itemize}
It is well-known that the retraction operator is a generalization of the exponential map \cite{absil2009optimization}. The iterative scheme of a Riemannian algorithm is usually given by
\[ x_{k+1} = R_{x_k}(t_k \eta_k), \]
where $\eta_k \in T_{x_k} \Mcal$ is a descent direction and $t_k > 0$ is a step size. 

For a compact submanifold $\Mcal$, the following Lipschitz-like property on the retraction operators is useful to establish the convergence of Riemannian algorithms. 
\begin{proposition}[\cite{boumal2019global}] \label{prop:lip-retr}
    Let $\Mcal$ be a compact submanifold and $R$ be any retraction operator. Then, there exists a constant $M_0 > 0$ such that, for any $x\in \Mcal$ and $u \in T_{x}\Mcal$, the following property holds 
    \be \label{eq:retr-lip} \| R_{x}(u) - x -u \|\leq M_0\|u\|^2. \ee
\end{proposition}

\subsection{Proximal smoothness} 

The notion of proximal smoothness, as introduced by \cite{clarke1995proximal}, refers to the characteristic of a closed set whereby the nearest-point projection becomes a singleton when the point is in close enough to the set.  This attribute is valuable in algorithmic design and theoretical analysis,  as it imbues nonconvex closed sets with a structure resembling that of convex sets.
 For any positive real number $\gamma$, we define the $\gamma$-tube around $\mathcal{M}$ as $
 U_{\mathcal{M}}(\gamma): = \{x:{\rm dist}(x,\mathcal{M}) <  \gamma\}.$ 
We say a closed set $\mathcal{M}$ is $\gamma$-proximally smooth if the projection operator $\Pcal_{\mathcal{M}}(x):=\argmin_{y \in \Mcal} \|y -x\|^2$ is a singleton whenever $x\in U_{\Mcal}(\gamma)$. Any closed and convex set is proximally smooth for arbitrary $\gamma \in [0, \infty)$. According to \cite[Corollary 4.6]{clarke1995proximal}, a closed set $\Mcal$ is convex if and only if it is proximally smooth with a radius of $\gamma$ for every $\gamma > 0$.
It is worth noting that that any compact $C^2$-submanifold of $\mathbb{R}^{d\times r}$ is a proximally smooth set \cite{clarke1995proximal,balashov2021gradient,davis2020stochastic}. For instance, the Stiefel manifold is a set that is $1$-proximally smooth. Throughout this paper, we assume that $\Mcal$ is $2\gamma$-proximally smooth. By  following the proof in \cite[Theorem 4.8]{clarke1995proximal}, a $2\gamma$-proximally smooth set $\mathcal{M}$ satisfies the following property: for any $\beta \in (0,2)$,
\be \label{eq:lip-proj-alpha}
\left\| \Pcal_{\mathcal{M}} (x) -\Pcal_{\mathcal{M}} (y)\right\| \leq \frac{2}{2 - \beta} \|x - y\|,~~ \forall x,y \in \bar{U}_{\mathcal{M}}(\beta \gamma), 
\ee
where $\bar{U}_{\Mcal}(\beta\gamma):=\{x: {\rm dist}(x,\Mcal) \leq \beta \gamma\}$ is the closure of $U_{\Mcal}(\beta\gamma)$. Moreover, for any point $x \in \mathcal{M}$ and a normal $v \in$ $N_{x} \mathcal{M}$, it holds that
\be \label{eq:normal-bound}
\iprod{v}{y-x} \leq \frac{\|v\|}{4\gamma} \|y-x\|^2, \quad \forall y \in \mathcal{M}, 
\ee
This is often referred to as the normal inequality \cite{clarke1995proximal,davis2020stochastic}. 
It is worth noting that for any closed convex set $\Mcal \subset \R^{d\times r}$, the projection operator $P_{\Mcal}$ is 1-Lipschitz continuous over $\R^{d\times r}$. Additionally,  the inequality \eqref{eq:normal-bound} holds with $\iprod{v}{y-x} \leq 0$. Therefore, the inequalities \eqref{eq:lip-proj-alpha} and \eqref{eq:normal-bound} can be considered as  generalizations from the closed convex set to the proximally smooth set.

\subsection{Euclidean consensus}
Observing that $\varphi^t(\bx) \geq 0$ for any $\bx \in \Mcal^N$, we can determine that the optimal solution set of \eqref{prob:original} is given by
\[ \mathcal{X}:=\{\bx \in \Mcal^N: x_1 =\cdots = x_N\}. \]
When the constraint set $\Mcal$ is convex, the problem \eqref{prob:original} reduces to the classic Euclidean consensus problem, which involves reaching consensus over a convex set. In this case, the projection of a point $\bx\in \Mcal^N$ onto $\mathcal{X}$ can be expressed as:
\[ \hat{\bx} := \argmin_{\by \in \mathcal{X}}\; \|\by - \bx\|^2 = [\hat{x}^\top, \ldots, \hat{x}^\top ]^\top, \]
where $\hat{x} = \frac{1}{N} \sum_{i=1}^N x_i$. 
Note that the Euclidean gradient of $\varphi^t(\bx)$ can be represented as $\nabla \varphi^t(\bx) = [\nabla \varphi^t_1(\bx)^\top, \ldots, \nabla \varphi^t_N(\bx)^\top]^\top =(I - \bW^t)\bx$, where $\nabla \varphi^t_i(\bx)=x_i - \sum_{j=1}W_{ij}^t x_j$.

Throughout the paper, we adopt the following assumptions on the mixing matrix $W$, which are commonly found in the literature, such as \cite{chen2021decentralized, zeng2018nonconvex}.

\begin{assumption} \label{assum-w}
     We assume that the mixing matrix $W$ satisfies the following conditions:
\begin{itemize}
    \item[(i)] $W_{ij}\geq 0$ for any $i, j\in [N]$ and $W_{ij}=0$  if and only if  $(i,j)\not\in {\mathcal E}$.
    \item[(ii)] $W = W^\top$ and $W \mathbf{1}_N = \mathbf{1}_N$.
    \item[(iii)] The null space of $(I-W)$ is $\operatorname{span}(\mathbf{1}_N)$.    
\end{itemize}
\end{assumption}
It should be noted that Assumption \ref{assum-w} implies that the second largest singular value $\sigma_2$ of $W$ lies in the interval $ [0,1)$ \cite{shi2015extra}. 

Consider the projected gradient descent with a unit step size, i.e., $\bx_{k+1} = P_{\Mcal^N}(\bx_k - \nabla \varphi^t(\bx_k)) = P_{\Mcal^N}(\bW^t \bx_k)$ with $P_{\Mcal^N}(\bx_k) := [P_{\Mcal}(\bx_{1,k})^\top, \ldots, P_{\Mcal}(\bx_{N,k})^\top]^\top$. Then, it holds that:
\be \label{eq:linear-ecu-lip} \begin{aligned}
\| \bx_{k+1} - \hat{\bx}_{k+1} \| & \leq \| \bx_{k+1} - \hat{\bx}_{k} \| \\
& = \| P_{\Mcal^N}(\bW^t \bx_k) - \hat{\bx}_k \| \\
& \leq \| \bW^t \bx_k - \hat{\bx}_k \| \\
& = \| (W^t - J)\otimes I_d (\bx_k - \hat{\bx}_k) \| \\
& \leq \sigma_2^t \|\bx_k - \hat{\bx}_k \|,
\end{aligned}
  \ee
where the first inequality follows directly from the definition of $\hat{x}_k$, the second inequality is from the 1-Lipschitz continuity of $P_{\Mcal^N}$, and the last inequality is obtained by the doubly stochastic property of $W$ with $\sigma_2$ being the second-largest singular value of $W$. Hence, the projected gradient descent with the unit step size has a linear convergence  with a rate of $\sigma_2^t$. 

The above analysis for establishing the linear convergence of the projected gradient descent is based on the 1-Lipschitz continuity of the projection operator. Another approach to construct the linear convergence rate in convex optimization is to  establish a certain regularity condition of \eqref{prob:original}, such as the restricted secant inequality, and then use this condition to  prove the convergence rate. In the paper \cite{chen2021local}, the authors establish the following restricted secant inequality:
\be \label{eq:rsi-euc} \iprod{\bx-\hat{\bx}}{\nabla \varphi^t(\bx)} \geq \frac{\mu_t L_t}{\mu_t + L_t}\|\bx - \hat{\bx}\|^2 + \frac{1}{\mu_t + L_t} \| \nabla \varphi^t(\bx) \|^2, \ee
where $L_t = 1 - \lambda_N(W^t)$ and $\mu_t = 1 -\lambda_2 (W^t)$ with $\lambda_N(W^t)$ and $\lambda_2(W^t)$ being the smallest eigenvalue and the second largest eigenvalue, respectively. Then, for the projected gradient descent with the constant step size $\alpha = 2/(\mu_t + L_t)$, i.e., 
 $\bx_{k+1} = P_{\Mcal^N}(\bx_k - \alpha \nabla \varphi^t(\bx_k))$, we have
\be \label{eq:2.6}
\begin{aligned}
    \| \bx_{k+1} - \hat{\bx}_{k+1} \|^2 & \leq \| \bx_{k+1} - \hat{\bx}_k  \|^2 \\
    & = \| P_{\Mcal^N}(\bx_{k} - \alpha \nabla \varphi^t(\bx_k)) -\hat{\bx}_k \|^2 \\ 
    & \leq \| \bx_{k} - \alpha \nabla \varphi^t(\bx_k) - \hat{\bx}_k \|^2 \\
    & = \| \bx_{k} - \hat{\bx}_k \|^2 + \alpha^2  \|\nabla \varphi^t(\bx_k)\|^2 - 2\alpha \iprod{\bx_{k} - \hat{\bx}_k}{\nabla \varphi^t(\bx_k)} \\
    & \leq \left( 1- \frac{2\alpha \mu_t L_t}{\mu_t + L_t} \right) \|\bx_k - \hat{\bx}_k \|^2 + \left( \alpha^2 - \frac{2\alpha}{\mu_t + L_t} \right) \| \nabla \varphi^t(\bx_k) \|^2 \\
    & =  \left(\frac{L_t - \mu_t}{L_t + \mu_t}\right)^2 \| \bx_k - \hat{\bx}_k \|^2.
\end{aligned} \ee
Here, the second inequality follows from the 1-Lipschitz continuity of $P_{\Mcal^N}$, and the last inequality is derived from the restricted secant inequality \eqref{eq:rsi-euc}. From \eqref{eq:2.6}, we conclude that the projected gradient descent converges linearly with a rate of $(L_t-\mu_t)/(L_t+\mu_t)$. If we set the step size to $1$, it can be shown that \cite[Appendix]{chen2021local}
\[ \| \bx_{k+1} - \hat{\bx}_{k+1} \| \leq \sigma_2^t \| \bx_{k} - \hat{\bx}_k \|. \]
This is consistent with the linear convergence  rate of $\sigma_2$ given by \eqref{eq:linear-ecu-lip}. It is worth noting that $(L_t-\mu_t)/(L_t+\mu_t) \leq \sigma_2^t$, indicating that
the analysis based on the restricted secant inequality provides a faster convergence rate with a step size of $\alpha = 2/(L_t + \mu_t)$. 

\section{Locally linear convergence of Riemannian gradient descent by local Lipschitz continuity} \label{sec:Lip}
The presence of a nonlinear and nonconvex manifold constraint in problem \eqref{prob:original} presents challenges when attempting to establish the global 1-Lipschitz continuity of $P_{\Mcal}$ and the restricted secant inequality, as depicted in \eqref{eq:rsi-euc}. Additionally, we have opted to employ the widely used Riemannian gradient descent instead of the projected gradient descent. It is worth noting that the projection operator can be viewed as a specific type of retraction \cite{absil2012projection}. Consequently, further exploration of the connections between general retraction operators and the projection operator is necessary.

Consider the Riemannian gradient descent method with a unit step size given by 
\be \label{eq:rgd} \bx_{k+1} = R_{\bx_k}(- \grad \varphi^t(\bx_k)), \ee
where $R_{\bx_k}(- \grad \varphi^t(\bx_k)):= [R_{x_{1,k}}(- \grad \varphi_1^t(\bx_k))^\top,\cdots,R_{x_{N,k}}(- \grad \varphi_N^t(\bx_k))]^\top$. 
For a set of $N$ points ${ x_1, \ldots, x_N } \subset \Mcal$, the induced arithmetic mean on the manifold is defined as
\[ \bar{x} := \argmin_{y \in \Mcal} \quad \sum_{i=1}^N \|y - x_i\|^2. \]
Through direct calculation, it can be shown that $\bar{x} = P_{\Mcal}(\hat{x})$ and $\bar{x}$ is a singleton if $\hat{x} \in U_{\Mcal} (2\gamma)$. Moreover, the projection of a point $\bx \in \mathbb{R}^{(Nd) \times r}$ onto $\mathcal{X}$ is given by
\[ \bar{\bx} = [\bar{x}^\top, \ldots, \bar{x}^\top]^\top.   \]

In \cite{deng2023decentralized}, it has been demonstrated that the projected gradient descent with a unit step size exhibits linear convergence under suitable initialization and for sufficiently large $t$. Let us now revisit the linear convergence result of the projected gradient descent as presented in \cite{deng2023decentralized}. The projected gradient update with the unit step size is given by 
\be\label{eq:consensus-pg-iter}
\bx_{k+1} = P_{\Mcal^N }(\bx_k - \nabla \varphi^t(\bx_k)).
\ee
 According to \cite[Theorem 3.1]{deng2023decentralized}, it is established that for any $\bx_k$ satisfying $\|\bar{x}_k - \hat{x}_k \| \leq \gamma/2$ and an appropriate value of $t$, the following inequality holds:
\be \label{eq:linear-pgd} \begin{aligned}
    \|\bx_{k+1} - \bar{\bx}_{k+1}\|  & \leq \|\bx_{k+1} - \bar{\bx}_k \|  = \| P_{\Mcal^N}(\bW^t \bx_k ) - \bar{\bx}_k \| \\ & \leq 2 \| \bW^t \bx_k - \hat{\bx}_k \| \leq 2 \sigma_2^t \| \bx_k - \bar{\bx}_k \|,
\end{aligned} \ee
where the first inequality is from the definition of $\bar{x}_k$, the second inequality is a consequence of the 2-Lipschitz continuity of  $P_{\Mcal}$ within $\bar{U}_{\Mcal} (\gamma)$, and the final equality results from the assumption regarding $W$. Consequently, if $2\sigma_2^t <1$, the sequence $\{x_k\}$ exhibits linear convergence towards the optimal solution set of problem \eqref{prob:original}.

The crux of the aforementioned analysis lies in the Lipschitz continuity of the projection operator $P_{\Mcal}$. To analyze the Riemannian gradient descent with such Lipschitz continuity, it is crucial to establish a relationship between a general retraction operator and the projection operator.
 
\begin{lemma} \label{lem:lip-proj}
    Let $R$ be any retraction on $\Mcal$. For any $x \in \mathcal{M}$ and $u \in \mathbb{R}^{d \times r}$, there exists a positive constant $M_1$ such that
    \be \label{eq:lip-proj}
    \left\| P_{\mathcal{M}}(x+u)- R_x(P_{T_x \mathcal{M}}(u))\right\| \leq M_1 \|u\|^2. 
    \ee
\end{lemma}
\begin{proof}
    According to \cite[Lemma 4.3]{deng2023decentralized}, there exists a positive constant $Q > 0$ such that for any $x \in \Mcal$ and $u \in \R^{d\times r}$, the following inequality holds:
    \[ \| P_{\Mcal}(x + u) - x - P_{T_x \Mcal}(u) \| \leq Q \|u\|^2. \]
   Taking into account the Lipschitz-like result of the retraction operator stated in Proposition \ref{prop:lip-retr}, we can derive the following:
    \[ \begin{aligned}
        \left\| P_{\mathcal{M}}(x+u)- R_x(\mathcal{P}_{T_x \mathcal{M}}(u))\right \| & \leq \left\| x - P_{T_x \Mcal}(u) - R_x(\mathcal{P}_{T_x \mathcal{M}}(u))\right \| + Q\|u\|^2  \\
        & \leq M_0 \| P_{T_x\Mcal} (u)\|^2 + Q\|u\|^2 \\
        & = (M_0 + Q) \| u \|^2.
    \end{aligned}
      \]
    This completes the proof by setting $M_1 = M_0 + Q$. 
\end{proof}

In contrast to the Euclidean consensus, the Lipschitz continuity of $P_{\Mcal}$, as described by \eqref{eq:lip-proj}, is only locally valid around $\Mcal$.  
Therefore, it becomes necessary to impose constraints on the initialization and $t$ to ensure that the Riemannian gradient descent with a unit step size remains within a small neighborhood.
For $\bx =[x_1^\top, \ldots, x_N^\top]^\top \in \R^{(Nd) \times r}$, let $\|\bx\|_{F, \infty} = \max_{i\in [N]} \|x_i\|$.  We begin by introducing the following lemma, which establishes a connection between the Riemannian gradient  $\sum_{i=1}^N \grad \varphi^t_i(\bx)$ and the consensus error $\|\bx_k - \bar{\bx}_k\|$. 

\begin{lemma} \label{lem:grad-consensus-sum2}
For any $\bx \in \Mcal^n$, it holds that
\be \label{eq:consensus-grad-sum2}
\left\| \sum_{i=1}^N \grad \varphi_i^t(\bx) \right \| \leq 2 \sqrt{N} L_p \|\bx - \bar{\bx} \|_{F,\infty} \|\bx - \bar{\bx}\|,  \ee
where $L_p := \max_{x \in {\rm conv}(\Mcal)} \| D P_{T_x\Mcal}(\cdot) \|_{\rm op}$ and $\|\cdot \|_{\rm op}$ is the operator norm.
\end{lemma}
\begin{proof}
    Note that $\sum_{i=1}^N \nabla \varphi_i^t(\bx) = \sum_{i=1}^N \left[  x_i - \sum_{j=1}^N W_{ij}^t x_j \right] = 0$. It holds that
    \[ \begin{aligned}
        \left\| \sum_{i=1}^N \grad \varphi_i^t(\bx) \right \| & = \left \|  \sum_{i=1}^N \left[ P_{T_{x_i}\Mcal}(\nabla \varphi_i^t(\bx)) - P_{T_{\bar{x}}\Mcal}(\nabla \varphi_i^t(\bx)) \right]   \right \| \\
        & \leq L_p \sum_{i=1}^N \left[ \| x_i - \bar{x}\| \|\nabla \varphi_i^t (\bx)\| \right] \\
        & \leq L_p \|\bx - \bar{\bx} \|_{F,\infty} \sum_{i=1}^N \| x_i - \sum_{j=1} W_{ij}^tx_j \| \\
        & \leq L_p \|\bx - \bar{\bx} \|_{F,\infty} \left[ \sum_{i=1}^N \| x_i -\bar{x}\| + \| \sum_{j=1} W_{ij}^tx_j -\bar{x} \| \right]\\
        & \leq 2 L_p \|\bx - \bar{\bx} \|_{F,\infty}\sum_{i=1}^N \| x_i - \bar{x} \| \\
        & \leq 2 \sqrt{N} L_p \|\bx - \bar{\bx} \|_{F,\infty} \|\bx - \bar{\bx}\|,
    \end{aligned}\]
    where the first inequality is due to the Lipschitz continuity of $P_{T_{x}\Mcal}(\cdot)$ over $x \in \Mcal$, the third inequality follows from the triangle inequality of $\|\cdot \|$, the fourth inequality uses the convexity property of $\|\cdot\|$, and the final inequality is derived from the inequality $\|a\|_1 \leq \sqrt{N}\|a\|$ that holds for any $a \in \R^{N}$. 
\end{proof}

It is worth noting that a tighter result of \eqref{eq:consensus-grad-sum2}, namely $\left\| \sum_{i=1}^N \grad \varphi_i^t(\bx) \right \| \leq L_t \|\bx - \bar{\bx}\|^2$, has been demonstrated in the specific scenario where $\Mcal$ corresponds to the Stiefel manifold, as shown in \cite[Lemma 10]{chen2021local}.   In contrast, in the general case of $\Mcal$, the explicit formulation of the projection operator remains unknown.

In addition, we require the utilization of the following lemma that pertains to the control of the Euclidean mean and the manifold mean, as presented in \cite[Lemma 4.4]{deng2023decentralized}.
\begin{lemma}[{\cite[Lemma 4.4]{deng2023decentralized}}] \label{lem:dist-consensus}
    There exists $M_2 > 0$ such that for any $\bx \in \Mcal^N$, 
    \be \| \hat{x} - \bar{x} \| \leq M_2 \frac{\| \bx - \bar{\bx} \|^2}{N}. \ee
\end{lemma}

With the above preparation, we are able to show that the iterates generated by \eqref{eq:rgd} will always stay in $\mathcal{N}_L:= \{\bx: \|\bx - \bar{\bx} \|_{F, \infty} \leq \delta_0 \}$ with $\delta_0 := \min\{\gamma, \sqrt{\gamma/(2M_2)}, 1/(16M_0 + 8M_1 + 4L_p), (1-2\sigma_2^t)/(8M_1\sqrt{N})\}$ if $x_0 \in \mathcal{N}_L$ and $t \geq \log_{\sigma_2} \frac{1}{4\sqrt{N}}$. 

\begin{lemma}\label{lemma:stay-neigb}
    Let $R$ be a retraction operator on $\Mcal$ and $\{ \bx_{k}\}$ be generated by \eqref{eq:rgd}. If $t \geq \log_{\sigma_2} \frac{1}{4\sqrt{N}} $  and $\bx_0 \in \mathcal{N}_L$, 
    then $\bx_{k} \in \mathcal{N}_L$ for any $k$.  
\end{lemma}
\begin{proof}
  We prove it by induction on $\|\bx_k - \bar{\bx}_k\|_{F,\infty}$. Suppose for some $k\geq 0$ that $\bx_k\in \mathcal{N}_L$.  Note that for any $i \in [N]$,
    \[  \begin{aligned}
        \| \sum_{j=1}^N W_{ij}^t x_{j,k} - \bar{x}_k  \| & \leq \| \sum_{j=1}^N W_{ij}^t x_{j,k} - \hat{x}_k \| + \| \hat{x}_k - \bar{x}_k \| \\
        & = \| \sum_{j=1}^N (W_{ij}^t - \frac{1}{N})(x_{j,k} - \hat{x}_k) \| + \| \hat{x}_k - \bar{x}_k \| \\
        & \leq \sum_{j=1}^N|W_{ij}^t - \frac{1}{N}| \|\bx_k - \bar{\bx}_k \|_{F, \infty} + \frac{M_2 \|\bx_k - \bar{\bx}_k\|^2}{N} \\
        & \leq \sqrt{N}\sigma_2^t \delta_0 + M_2 \delta_0^2 \\
        & \leq \gamma,
    \end{aligned} \]
    where the first inequality is from the triangle inequality, the second equality is due to the Cauchy inequality, and the third inequality comes  from the bound on the distance between the row of $W^t$ and $\mathbf{1}/N$ \cite{diaconis1991geometric,boyd2004fastest}. This implies that $\sum_{j=1}^N W_{ij}^t x_{j,k} \in \bar{U}_{\Mcal}(\gamma)$ for any $i\in [N]$. Moreover, it follows from $\bx_k\in \mathcal{N}_L$ and the definition of $\mathcal{N}_L$ that $\|x_{i,k} - \bar{x}_k\| \leq \gamma, \forall i\in [N]$ and  
    \be \label{eq:xhat} \|\hat{x}_k - \bar{x}_k \| \leq \frac{1}{N}\sum_{i=1}^N \|x_{i,k} - \bar{x}_k\| \leq \gamma. \ee
    This gives $\hat{x}_k \in \bar{U}_{\Mcal}(\gamma)$. Then, we have for any $i \in [N]$,
    \be \label{eq:estimate-xkp1-xk}
    \begin{aligned}
         \| x_{i, k+1} - \bar{x}_{k} \|  = &  \| R_{x_{i,k}}(-\grad \varphi^t_i(\bx_k)) - \bar{x}_k \| \\
        \leq & \| P_{\Mcal}(\sum_{j=1}^N W_{ij}^t x_{j,k}) - \bar{x}_k \| +  M_1 \|\nabla \varphi^t_i(\bx_k)\|^2  \\
        \leq & 2 \| \sum_{j=1}^N W_{ij}^t x_{j,k} - \hat{x}_k \| + M_1\| x_{i,k} - \sum_{j=1}^N W_{ij}^t x_{j,k} \|^2 \\
        \leq & 
        2\sqrt{N}\sigma_2^t \| \bx_k - \bar{\bx}_k \|_{F,\infty} + 4 M_1 \| \bx_k - \bar{\bx}_k \|_{F,\infty}^2 \\
        \leq & \frac{1}{2} \|\bx_k - \bar{\bx}_k \|_{F, \infty} + \frac{1}{2}\|\bx_k - \bar{\bx}_k \|_{F, \infty} \\
        =&  \|\bx_k - \bar{\bx}_k \|_{F, \infty},
    \end{aligned}
     \ee
    where the first inequality is from Lemma \ref{lem:lip-proj}, the second inequality is due to the 2-Lipschitz continuity of $P_{\Mcal}$ over $\bar{U}_{\Mcal}(\gamma)$, the third inequality is from the assumptions on $W$, and the last inequality is from the definition of $\mathcal{N}_L$. This gives 
    \[ {\rm dist}(\hat{x}_{k+1}, \Mcal) \leq \frac{1}{N} \sum_{i=1}^N \| x_{i,k+1} - \bar{x}_k \| \leq \gamma. \]
    Furthermore, we have 
    \[ \begin{aligned}
        & \| x_{i, k+1} - \bar{x}_{k+1} \|\\
        \leq& \| x_{i,k+1} - \bar{x}_k \| + \| \bar{x}_k - \bar{x}_{k+1} \| \\
         \leq & 2\sqrt{N}\sigma_2^t \| \bx_k - \bar{\bx}_k \|_{F,\infty} + 4 M_1 \| \bx_k - \bar{\bx}_k \|_{F,\infty}^2  + 2 \|\hat{x}_k - \hat{x}_{k+1}\|  \\
        = & 2\sqrt{N}\sigma_2^t \| \bx_k - \bar{\bx}_k \|_{F,\infty} + 4 M_1 \| \bx_k - \bar{\bx}_k \|_{F,\infty}^2  + 2 \| \frac{1}{N} \sum_{i=1}^N R_{x_{i,k}}(- \grad \varphi^t_i(\bx_k)) - \hat{x}_k \| \\
        \leq & 2\sqrt{N}\sigma_2^t \| \bx_k - \bar{\bx}_k \|_{F,\infty} + 4 M_1 \| \bx_k - \bar{\bx}_k \|_{F,\infty}^2 + 2 \|\frac{1}{N}\sum_{i=1}^N \grad \varphi^t_i(\bx_k)\| \\
         & + \frac{2M_0}{N} \|\grad \varphi^t(\bx_k) \|^2 \\
        \leq & 2\sqrt{N}\sigma_2^t \| \bx_k - \bar{\bx}_k \|_{F,\infty} + 4M_1\| \bx_k - \bar{\bx}_k \|_{F,\infty}^2  + 2L_p /\sqrt{N}\|\bx_k - \bar{\bx}_k\|_{F, \infty} \|\bx_{{k}} -\bar{\bx}_{{k}}\| \\
        & + \frac{8M_0}{N}\|\bx_{{k}} -\bar{\bx}_{{k}}\|^2 \\
        \leq & \frac{1}{2}\|\bx_k - \bar{\bx}_k\|_{F,\infty} + (4M_1  + 2L_p  + 8M_0 )\|\bx_k - \bar{\bx}_k\|_{F,\infty}^2 
        \\
        \leq & \|\bx_k - \bar{\bx}_k \|_{F, \infty},
    \end{aligned}
     \]
     where the first inequality is from the triangle inequality, the second inequality is due to \eqref{eq:estimate-xkp1-xk} and the 2-Lipschitz continuity of $P_{\Mcal}$ over $\bar{U}_{\Mcal}(\gamma)$, the third inequality comes from Lemma \ref{lem:lip-proj} and Lemma \ref{prop:lip-retr}, the fourth inequality is from Lemma \ref{lem:grad-consensus-sum2}, the 2-Lipschitz continuity of $P_{\Mcal}$ over $\bar{U}_{\Mcal}(\gamma)$ and the assumptions on $W$, and the last inequality is from the definition of $\mathcal{N}_L$. Hence, if $\bx_k \in \mathcal{N}_L$, then $\bx_{k+1} \in \mathcal{N}_L$. We complete the proof.
\end{proof}

The aforementioned lemma provides a sufficient condition, namely $\bx_0 \in \mathcal{N}_L$, to ensure that $\sum_{j=1} W_{ij}^t x_{jk} \in \bar{U}_{\Mcal}(\gamma)$ and $\hat{x}_k \in \bar{U}_{\Mcal}(\gamma)$ for any $k \geq 0$. 
Consequently, the Riemannian gradient descent with the unit step size \eqref{eq:rgd} exhibits linear convergence if
 $x_0 \in \mathcal{N}_L$ and $t \geq \log_{\sigma_2} \frac{1}{4\sqrt{N}}$.

\begin{theorem} \label{thm:linear-lip}
    Let $R$ be a retraction operator on $\Mcal$ and $\{\bx_k\}$ be generated by \eqref{eq:rgd}. If $t \geq \log_{\sigma_2} \frac{1}{4\sqrt{N}} $ and $\bx_0 \in \mathcal{N}_L$, 
    then $\{\bx_{k}\}$ converges Q-linearly to the solution set $\mathcal{X}$ with rate $(1+2\sigma_2^t)/2$,  which is characterized by the inequality   
    \[ \| \bx_{k+1} - \bar{\bx}_{k+1} \| \leq \frac{1+2\sigma_2^t}{2}  \| \bx_k - \bar{\bx}_k \| \leq \left(\frac{1+2\sigma_2^t}{2}\right)^{k+1}\|\bx_0 - \bar{\bx}_0 \|. \]  
\end{theorem}
\begin{proof}
   With  Lemmas \ref{lem:lip-proj} and \ref{lemma:stay-neigb}, we can derive the following results:
    \be \label{eq:linear-lip-ana} \begin{aligned}
        \| \bx_{k+1} - \bar{\bx}_{k+1} \| & \leq \| R_{\bx_k}(- \grad \varphi^t(\bx_k)) - \bar{\bx}_k  \| \\
        & \leq \| P_{\mathcal{M}^N}(\bx_k - \nabla \varphi^t(\bx_k)) - \bar{\bx}_k \| + 4 M_1 \|\bx_k  - \bar{\bx}_k\|^2 \\
        & \leq 2 \| \bW^t \bx_k - \hat{\bx}_k \|  + 4M_1 \|\bx_k  - \bar{\bx}_k\|^2 \\
        & \leq 2 \sigma_2^t \| \bx_k - \bar{\bx}_k \| + 4M_1 \|\bx_k  - \bar{\bx}_k\|^2 \\
        & \leq (1 + 2\sigma_2^t)/2 \| \bx_k  - \bar{\bx}_k  \|,
    \end{aligned} \ee
    where the third inequality is from the $2$-Lipschitz continuity of $P_\Mcal$ over $\bar{U}_{\Mcal}(\gamma)$, $\sum_{j=1} W_{ij}^t x_{jk} \in \bar{U}_{\Mcal}(\gamma)$, and $\hat{x}_k \in \bar{U}_{\Mcal}(\gamma)$,  and the last inequality is due to $ 4M_1 \|\bx_k  - \bar{\bx}_k\| \leq 4M_1 \sqrt{N} \|\bx_k  - \bar{\bx}_k\|_{F,\infty} \leq (1-2\sigma_2^t)/2$. The proof is completed. 
\end{proof}

\begin{remark}
    The rate of convergence $(1+ 2\sigma_2^t)/2$ in comparison to its Euclidean counterpart $\sigma_2^t$ may seem worse.   However, as $\bx_k$ converges to $\mathcal{X}$,  we can still achieve the rate of $\sigma_2^t$ for the Riemannian gradient descent with the unit step size. 
    Specifically, by utilizing \eqref{eq:lip-proj-alpha}, the $2$-Lipschitz continuity used in the third inequality of \eqref{eq:linear-lip-ana} can be improved to $2/(2-\beta_k)$ with $\beta_k \rightarrow 0$. Consequently, we obtain the following inequality:
    \[ \| \bx_{k+1} - \bar{\bx}_{k+1}\| \leq \left(\frac{2}{2-\beta_k} \sigma_2^t + 4M_1 \|\bx_k -\bar{\bx}_k\|\right) \| \bx_k -\bar{\bx}_k\| \|\rightarrow \sigma_2^t\|\bx_k - \bar{\bx}_k\|. \]
    
\end{remark}

\section{Restrict secant inequality and locally linear convergence of Riemannian gradient descent} \label{sec:rsi}

In this section, we will begin by examining the restricted secant inequality of \eqref{prob:original}. Subsequently, we will use such restricted secant inequality to analyze the convergence behavior of the Riemannian gradient descent with a constant step size $\alpha > 0$, given by:
\be \label{eq:rgd-alpha} \bx_{k+1} = R_{\bx_k}(-\alpha \grad \varphi^t(\bx_k)). \ee

\subsection{Restricted secant inequality}
Firstly, we can establish the following inequality relating the Euclidean consensus error, denoted as $\|\bx - \hat{\bx}\|$, to the manifold consensus error, denoted as $\| \bx - \bar{\bx} \| $.
\begin{lemma} \label{lem:conseus-bound}
For any $\bx\in \Mcal^N$ with $\| \bx - \bar{\bx} \|^2 \leq N/(4M_2^2)$, we have
   \be \label{eq:consensus-ineq} \frac{1}{4}\| \bx - \bar{\bx} \|^2 \leq \| \bx - \hat{\bx}  \|^2 \leq \| \bx - \bar{\bx} \|^2. \ee
\end{lemma}
\begin{proof} According to the definitions of $\bar{x}$ and $\hat{x}$, we have $\| \bx - \hat{\bx} \|^2 \leq \|\bx - \bar{\bx} \|^2$. Further, we have
    \be \label{eq:est-con}
    \begin{aligned}
        \| \bx - \bar{\bx}  \|^2 & = \sum_{i=1}^N \|x_i -\hat{x} + \hat{x} -\bar{x}\|^2 \\
        & \leq 2\sum_{i=1}^N \left( \|x_i - \hat{x}\|^2 + \|\hat{x} - \bar{x}\|^2 \right) \\
        & \leq 2 \|\bx - \hat{\bx}\|^2 + 2N \cdot  M_2^2 \frac{\|\bx - \bar{\bx}\|^4}{N^2} \\
        & \leq 2 \|\bx - \hat{\bx}\|^2 + \frac{1}{2}\| \bx - \bar{\bx} \|^2,
    \end{aligned} \ee
    where the first inequality is due to $(a+b)^2 \leq 2(a^2 + b^2)$, the second inequality is from Lemma \ref{lem:dist-consensus}, and the last inequality comes from $\| \bx - \bar{\bx} \|^2 \leq N/(4M_2^2)$. This gives $\|\bx - \bar{\bx} \| / 4 \leq \|\bx - \hat{\bx}\|^2$. We complete the proof.
\end{proof}

By utilizing the orthogonality structure of the Stiefel manifold, we can establish a stronger inequality than \eqref{eq:est-con}, namely $\|\bx - \bar{\bx}\|^2/2 \leq \| \bx - \hat{\bx}\|^2 \leq \|\bx - \bar{\bx} \|^2$, which holds for any $\bx \in {\rm St}(d,r)^N$ as shown in \cite[Lemma 1]{chen2021local}. Additionally, by \cite[Lemma 8]{chen2021local}, we have the following lemma regarding the quadratic growth of the objective function $\varphi^t(\bx)$. 
\begin{lemma} \label{lemma:quad-growth}
For any $t \geq 1$ and $\bx \in \mathcal{M}^N$, it holds that
\be \label{eq:qg-1}
\varphi^t(\bx)-\varphi^t(\bar{\bx}) \geq \frac{\mu_t}{2}\|\bx-\hat{\bx}\|^2.
\ee
Moreover, if $\|\bx-\bar{\bx}\|^2 \leq \frac{N}{4M_1^2}$, we have
\be \label{eq:qg-2}
\varphi^t(\bx)-\varphi^t(\bar{\bx}) \geq \frac{\mu_t}{4}\|\bx - \bar{\bx}\|^2. 
\ee
\end{lemma}
\begin{proof}
    By \cite[Lemma 8]{chen2021local}, it holds that $2\varphi^t(\bx) \geq \mu_t\|\bx-\hat{\bx}\|^2.$
    By taking into account the fact that $\varphi^t(\bar{\bx}) =0$, we can conclude that \eqref{eq:qg-1} holds. The inequality \eqref{eq:qg-2} is derived by applying Lemma \ref{lem:conseus-bound} to \eqref{eq:qg-1}.
\end{proof}

With the aforementioned preparation, we are now ready to establish the following type of restricted secant inequality.
\begin{lemma}
    Suppose that $\mathbf{x} \in \Mcal^N$ with $\|\bx - \bar{\bx}\|_{F, \infty}^2 \leq \gamma $ and $ \|\bx - \bar{\bx}\|^2 \leq \frac{N}{4M_2^2}$. For any $t\geq 1$, 
    the following holds:
    \be \label{eq:RSI-R}
    \left\langle\bx-\bar{\bx}, \operatorname{grad} \varphi^t(\bx)\right\rangle \geq (1 - \frac{\|\bx - \bar{\bx}\|_{F, \infty}^2}{\gamma})\frac{\mu_t}{4}\|\bx - \bar{\bx}\|^2.
    \ee
\end{lemma}
\begin{proof}
It follows from the definition of $\grad \varphi^t(\bx)$ that
\be \label{eq:grad-prod} \iprod{\grad \varphi^t(\bx)}{\bx - \bar{\bx}} = \iprod{\nabla \varphi^t(\bx)}{\bx - \bar{\bx}} - \iprod{P_{N_{\bx} \Mcal^N}(\nabla \varphi^t(\bx))}{\bx - \bar{\bx}}. \ee
Since $\Mcal$ is $2\gamma$-proximally smooth, we have by \eqref{eq:normal-bound} that
\be \label{eq:grad-decomp-normal}
\begin{aligned}
  \iprod{P_{N_{\bx} \Mcal^N}(\nabla \varphi^t(\bx))}{\bx - \bar{\bx}} & = \iprod{P_{N_{\bx} \Mcal^N}(\bx - \bar{\bx})}{\nabla \varphi^t(\bx)}   \\
  & = \sum_{i=1}^N \iprod{P_{N_{x_i}\Mcal}(x_i - \bar{x})}{\sum_{j=1}^N W_{ij}^t (x_i - x_j)} \\
 & =   \sum_{i=1}^N \sum_{j=1}^N W_{ij}^t \iprod{P_{N_{x_i}\Mcal}(x_i - \bar{x})}{ x_i - x_j}  \\
  & \leq  \frac{\max_i \|x_i - \bar{x}\|^2}{4\gamma}\sum_{i=1}^N  \sum_{j=1}^N W_{ij}^t \|x_i - x_j\|^2 \\
  & \leq \frac{\|\bx - \bar{\bx}\|_{F, \infty}^2}{\gamma} \varphi^t(\bx).
\end{aligned}
 \ee
Substituting \eqref{eq:grad-decomp-normal} into \eqref{eq:grad-prod} results in the following inequality:
 \be \label{eq:grad-decomp} \iprod{\grad \varphi^t(\bx)}{\bx - \bar{\bx}} \geq (1-\frac{\|\bx - \bar{\bx}\|_{F, \infty}^2}{\gamma}) \varphi^t(\bx).  \ee
By combining \eqref{eq:grad-decomp} and Lemma \ref{lemma:quad-growth}, we obtain \eqref{eq:RSI-R}. 
\end{proof}

To establish a more general restricted secant inequality, we also provide the following lemma that relates the objective function values, gradients, and consensus error.
\begin{lemma} \label{lem:grad-consensus-sum}
For any $\bx \in \Mcal^N$, it holds that
\begin{align}
    & \| \grad \varphi^t(\bx) \|^2 \leq 2L_t \varphi^t(\bx), \label{eq:grad-obj} \\
    & \| P_{N_{x_i}\Mcal}(\nabla \varphi_i^t(\bx)) \| \leq \sqrt{\frac{1}{\gamma}} \|\bx - \bar{\bx}\|_{F, \infty}^{\frac{3}{2}}. \label{eq:grad-normal}
\end{align}
\end{lemma}
\begin{proof}
    Note that $\nabla \varphi^t(\bx) = (I - \bW^t)\bx$ is Lipschitz continuous with modulus $L_t$. Then, it holds that
    \[ \begin{aligned}
        \varphi^t(\bx - \frac{1}{L_t} \nabla \varphi^t(\bx)) & \leq \varphi^t(\bx) + \iprod{\nabla \varphi^t(\bx)}{-\frac{1}{L_t} \nabla \varphi^t(\bx)} + \frac{1}{2L_t}\|\nabla \varphi^t(\bx)\|^2 \\
        & = \varphi^t(\bx) - \frac{1}{2L_t}\|\nabla \varphi^t(\bx)\|^2.
    \end{aligned}
      \]
    By the fact that $\varphi^t(\bx - \frac{1}{L_t} \nabla \varphi^t(\bx)) \geq 0$,
    we have
    \[ \begin{aligned}
        \| \grad \varphi^t(\bx) \|^2  \leq \|\nabla \varphi^t(\bx)\|^2 
         \leq 2 L_t \varphi^t(\bx). 
    \end{aligned} 
    \]
    By \eqref{eq:normal-bound}, one has
    \[ \begin{aligned}
        \| P_{N_{x_i}\Mcal}(\nabla \varphi_i^t(\bx_k))\|^2 & = \iprod{P_{N_{x_i}\Mcal}(\nabla \varphi_i^t(\bx_k))}{P_{N_{x_i}\Mcal}(\nabla \varphi_i^t(\bx_k))} \\
        & = \iprod{P_{N_{x_i}\Mcal}(\nabla \varphi_i^t(\bx_k))}{\sum_{j=1}^N W_{ij}^t(x_i - x_j)} \\
        & \leq \sum_{j=1}^N W_{ij}^t \frac{\| P_{N_{x_i}\Mcal}(\nabla \varphi_i^t(\bx_k))\|}{4\gamma} \|x_i - x_j\|^2 \\
        & \leq \frac{ \|\bx-\bar{\bx}\|_{F,\infty}}{4\gamma} \cdot (4  \| \bx - \bar{\bx} \|_{F, \infty}^2) \\
        & = \frac{1}{\gamma}  \| \bx - \bar{\bx} \|_{F, \infty}^3,
    \end{aligned} \]
    which gives \eqref{eq:grad-normal}. The proof is completed.
\end{proof}

The estimate \eqref{eq:grad-obj} coincides with the result for the Stiefel manifold \cite[Lemma 10]{chen2021local}.  It is worth noting that for any $x\in {\rm St}(d,r)$ and $y\in \R^{d\times r}$, we have $P_{N_x {\rm St}(d,r)}(y) = x(x^\top y + y^\top x)/2$. In the case where $\Mcal$ is the Stiefel manifold, we can obtain a stronger version of \eqref{eq:grad-normal}, which is given by:
\[ \begin{aligned}
    P_{N_{x_i}\Mcal}(\nabla \varphi^t_i(\bx_k)) & = x_i(x_i^\top (x_i - \sum_{j=1}^N W_{ij}^t x_j ) + (x_i - \sum_{j=1}^N W_{ij}^t x_j )^\top x_i)/2 \\
    & = x_i \sum_{j=1}^NW_{ij}^t(x_i - x_j)^\top (x_i - x_j)  \\
    & \leq \|\bx - \bar{\bx}\|_{F, \infty}^2.
\end{aligned}
  \]

With the above lemma, we present a more general form of the restricted secant inequality.
\begin{theorem} \label{thm:rsi}
    For any $\nu \in [0,1]$, $t \geq 1$, and any $x\in \mathcal{M}^N$ with $\|\bx - \bar{\bx}\|_{F, \infty}^2 \leq \gamma $ and $ \|\bx - \bar{\bx}\|^2 \leq \frac{N}{4M_2^2}$,
\begin{align}
    \iprod{\bx - \bar{\bx}}{\grad \varphi^t(\bx)} & \geq \frac{\Phi}{2L_t}\|\grad \varphi^t(\bx)\|^2, \label{eq:rsi-1}   \\
     \iprod{\bx - \bar{\bx}}{\grad \varphi^t(\bx)} & \geq \nu\cdot \frac{\Phi}{2L_t}\|\grad \varphi^t(\bx)\|^2 + (1-\nu)\gamma_R \|\bx - \bar{\bx}\|^2, \label{eq:rsi}
\end{align}
where $\Phi := 2\left(1 - \frac{\|\bx - \bar{\bx}\|_{F,\infty}^2}{2\gamma} \right) > 1$ and $\gamma_R := (1 - \frac{\|\bx - \bar{\bx}\|_{F, \infty}^2}{\gamma})\frac{\mu_t}{4}$.
\end{theorem}

\begin{proof}
    By \eqref{eq:grad-decomp} and \eqref{eq:grad-obj}, we have
    \[ \iprod{\bx - \bar{\bx}}{\grad \varphi^t(\bx)} \geq 2\left(1 - \frac{\|\bx - \bar{\bx}\|_{F, \infty}^2}{2 \gamma}\right) \frac{\|\grad \varphi^t(\bx)\|^2}{2L_t}. \]
    This gives \eqref{eq:rsi-1}. Combining the above inequality with \eqref{eq:RSI-R} yields \eqref{eq:rsi}.   
\end{proof}

\subsection{Locally linear convergence by the restricted secant inequality}
To prove the locally linear convergence of the Riemannian gradient descent \eqref{eq:rgd-alpha}, 
 our first step is to demonstrate that all the iterates remain within a neighborhood around $\mathcal{X}$ given appropriate initialization and step size. Subsequently, by utilizing the restricted secant inequality within this region, we establish the linear convergence rate.

Let us define the neighborhood 
 \be \label{eq:neighborhood} \mathcal{N}_{R}:= \mathcal{N}_{1} \cap \mathcal{N}_{2},\ee
 where
\begin{eqnarray}
     \mathcal{N}_{1} & := &\{ \bx : \| \bx -\bar{\bx}\|^2 \leq N \delta_{1}^2 \}, \\ \label{eq:neigh-R-1}
     \mathcal{N}_{2} & := &\{ \bx :  \| \bx -\bar{\bx}\|_{F,\infty} \leq \delta_{2} \}, \label{eq:neigh-R-2}
\end{eqnarray}
 and $\delta_{1}, \delta_{2}$ satisfy 
 \begin{align}
     \delta_{1} & \leq \min \left \{\frac{\delta_{2}}{4}, \frac{1}{2M_2 + 2L_p \sqrt{N} + 2L_t^2}, \sqrt{\frac{\gamma}{M_2}}, \frac{1}{2M_1\gamma} \right \}, \label{eq:delta-1} \\
     \delta_{2} & \leq \min\left\{\frac{1}{32}, \gamma, \frac{\sqrt{2\gamma}}{6}\right\}. \label{eq:delta-2}
 \end{align}

We will now establish the following lemma regarding the distance between $\bar{x}_k$ and $\bar{x}_{k+1}$, which plays a crucial role in characterizing the local neighborhood.
\begin{lemma} \label{lem:dist-manifold-mean}
    Let $\{\bx_k\}$ be the sequence generated by \eqref{eq:rgd-alpha}. For any $\bx_k, \bx_{k+1} \in \mathcal{N}_{1}$, it holds that 
    \be \label{eq:dist-manifold-mean}  
    \| \bar{x}_{k+1} - \bar{x}_{k}\| \leq  (2L_t^2 M_0 \alpha^2 + 2L_p \sqrt{N} \alpha) \delta_{1}^2. 
    \ee
\end{lemma}
\begin{proof}
It follows from Lemma \ref{lem:dist-consensus} that $\|\hat{x}_k - \bar{x}_k\| \leq M_2\|\bx_k - \bar{\bx}_k\|^2/N \leq M_2 \delta_{1}^2$. Hence, by \cite[Lemma 3.1]{davis2020stochastic} and noting $M_2 \delta_{1}^2 <\gamma $, we have
\[ \begin{aligned}
   & \| \bar{x}_{k+1} - \bar{x}_k \| \leq 2 \|\hat{x}_{k+1} - \hat{x}_k \| \\
   = & 2 \|\frac{1}{N}\sum_{j=1}^N (x_{i,k+1} - x_{i,k}) \| \\
   \leq & 2 \|\frac{1}{N}\sum_{j=1}^N (x_{i,k+1} - x_{i,k} + \alpha \grad \varphi_i^t(\bx_k) ) \| + 2\| \frac{\alpha}{N} \sum_{i=1}^N \grad \varphi_i^t(\bx_k) \| \\
   \leq & \frac{2 M_0 \alpha^2}{N} \sum_{i=1}^N \|\grad \varphi_i^t(\bx_k) \|^2 + 2\| \frac{\alpha}{N} \sum_{i=1}^N \grad \varphi_i^t(\bx_k) \| \\
   \leq & \frac{2L_t^2 M_0 \alpha^2 + 2L_p \sqrt{N} \alpha}{N} \| \bx_k - \bar{\bx}_k\|^2\\
   \leq & (2L_t^2 M_0 \alpha^2 + 2L_p \sqrt{N} \alpha) \delta_{1}^2,
   \end{aligned} \]
   where the first inequality follows from the 2-Lipschitz continuity of $P_{\Mcal}$ over $\bar{U}_{\Mcal}(\gamma)$, the second inequality is from the triangle inequality, and the third inequality is derived from Proposition \ref{prop:lip-retr}.
\end{proof}

The aforementioned result generalizes the corresponding lemma for the Stiefel manifold presented in \cite[Lemma 12]{chen2021local}. Consequently, we can now demonstrate that the iterates produced by \eqref{eq:rgd-alpha} remain within $\mathcal{N}_R$ given sufficiently large $t$ and an appropriately chosen $\alpha$.
\begin{lemma} \label{lem:stay-neighborhood-rsi}
    Let $\alpha \leq \min \left\{ \frac{\nu \Phi \tilde{\gamma}_{R} (1-2 \tilde{\gamma}_{R})}{L_t(1-\tilde{\gamma}_{R})((1-2\tilde{\gamma}_{R})M_0^2L_t^2N\delta_{1}^2 + \tilde{\gamma}_{R})} ,1, \frac{1}{M_0} \right \}$ with $\tilde{\gamma}_{R}:= (1 -\nu) \gamma_{R}$ and $t \geq \log_{\sigma_2} \frac{1}{2\sqrt{N}}$. If $\bx_k \in \mathcal{N}_{R}$, then $\bx_{k+1} \in \mathcal{N}_{R}$. 
\end{lemma}
\begin{proof} 
    By the definition of $\mathcal{N}_R$, it holds that 
    \be  \label{eq:NR-rsi}\|\bx - \bar{\bx}\|^2_{F, \infty} \leq \gamma, \quad {\rm and} \quad \|\bx - \bar{\bx}\|^2 \leq \frac{N}{4M_2^2}. \ee
    Then, we have
    \be \label{eq:contrac-xk} \begin{aligned}
        & \|\bx_{k+1} - \bar{\bx}_{k+1}\|^2  \leq \| \bx_{k+1} - \bar{\bx}_k \|^2 \\
        = & \sum_{i=1}^N \| R_{x_{i,k}}(-\alpha \grad \varphi_i^t(\bx)) - \bar{x}_k  \|^2 \\
        \leq & \sum_{i=1}^N(1+ \frac{1}{\beta}) \| x_{i,k} - \alpha \grad \varphi_i^t(\bx_k) -\bar{x}_k \|^2 + (1+ \beta) (M_0 \|\alpha \grad \varphi^t(\bx_k) \|^2)^2 \\
        \leq & (1+ \frac{1}{\beta}) \left[ \|\bx_k - \bar{\bx}_k\|^2 - 2\alpha \iprod{\grad \varphi^t(\bx_k)}{\bx_k - \bar{\bx}_k} \right] \\ 
        & + \left((1+\beta)M_0^2\|\alpha \grad \varphi^t(\bx_k)\|^2 + 1 + \frac{1}{\beta} \right)\alpha^2\| \grad \varphi^t(\bx_k)\|^2 \\
        \leq & (1+\frac{1}{\beta})(1-2\alpha(1-\nu)\gamma_R) \|\bx_k - \bar{\bx}_k\|^2 \\
        & + \left(\left((1+\beta)M_0^2\alpha^2 L_{t}^2 N \delta_{1}^2 + 1 + \frac{1}{\beta} \right)\alpha^2 -  \frac{\alpha \nu\Phi}{L_t} \right) \| \grad \varphi^t(\bx_k)\|^2,
    \end{aligned} \ee
    where the second inequality is from \eqref{eq:retr-lip} and $(a+b)^2 \leq (1+1/\beta)a^2 + (1+\beta)b^2$ for any $a,b,\beta > 0$, the third inequality is due to \eqref{eq:rsi}, and the last inequality comes from $\| \grad \varphi^t(\bx_k) \|^2 \leq L_t \| \bx_k - \bar{\bx}_k \|^2$ and $\bx_k \in \mathcal{N}_{1}$. Let 
    $\beta = (1 - 2\alpha(1-\nu) \gamma_R)/ (\alpha(1-\nu) \gamma_R)$
    and $\alpha \leq \min\left \{\frac{\nu \Phi \tilde{\gamma}_R (1-2 \tilde{\gamma}_R)}{L_t(1-\tilde{\gamma}_R)((1-2\tilde{\gamma}_R)M_0^2L_t^2N\delta_{1}^2 + \tilde{\gamma}_R)} , 1\right\}$, 
    we have
    \[   \|\bx_{k+1} - \bar{\bx}_{k+1}\|^2  \leq \| \bx_{k} - \bar{\bx}_k \|^2.  \]
    This implies that $\bx_{k+1}\in \mathcal{N}_{1}$. For each $i \in [N]$, it holds
    \[  \begin{aligned}
       & \|x_{i,k+1} -  \bar{x}_k \|  \\
       \leq & \| x_{i,k} - \alpha \grad \varphi_i^t(\bx_k) - \bar{x}_k \| + M_0 \| \alpha \grad \varphi_i^t(\bx_k)\|^2 \\
       = & \| (1-\alpha)(x_{i,k}-\bar{x}_k) + \alpha (\hat{x}_k - \bar{x}_k) + \alpha \sum_{j=1}^N W_{ij}^t(x_{j,k} -\hat{x}_k) + \alpha P_{N_{x_{i,k}}\Mcal}(\nabla \varphi_i^t(\bx_k)) \| \\
       & + M_0 \| \alpha \grad \varphi_i^t(\bx_k)\|^2 \\
       \leq & (1-\alpha) \delta_{2} + M_2 \alpha \delta_{1}^2 + \frac{\alpha}{2}\delta_{2} + \frac{\alpha}{\sqrt{\gamma}} \delta_{2}^{\frac{3}{2}} +  4M_0\alpha^2 \delta_{2}^2   \\
       = & (1-\frac{\alpha}{2}) \delta_{2} + \frac{\alpha}{\sqrt{\gamma}}\delta_{2}^{\frac{3}{2}} + 4M_0 \alpha^2 \delta_{2}^2 + M_2 \alpha \delta_{1}^2,
    \end{aligned}  \]
    where the first inequality is from \eqref{eq:retr-lip} and the second inequality is due to \eqref{eq:grad-obj} and \eqref{eq:grad-normal}. By Lemma \ref{lem:dist-manifold-mean}, we have
    \[ \begin{aligned}
        & \|x_{i,k+1} - \bar{x}_{k+1} \| \leq \|x_{i,k+1} - \bar{x}_k\| + \| \bar{x}_k - \bar{x}_{k+1} \| \\
        \leq & (1-\frac{\alpha}{2}) \delta_{2} + \frac{\alpha}{\sqrt{\gamma}}\delta_{2}^{\frac{3}{2}} + 4M_0 \alpha^2 \delta_{2}^2 + M_2 \alpha \delta_{1}^2 + (2L_t^2 M_0 \alpha^2 + 2L_p \sqrt{N} \alpha) \delta_{1}^2 \\
        = & (1-\frac{\alpha}{2}) \delta_{2} + \frac{\alpha}{\sqrt{\gamma}}\delta_{2}^{\frac{3}{2}} + 4M_0 \alpha^2 \delta_{2}^2 + (M_2 + 2L_p\sqrt{N}  + 2L_t^2 M_0^2 \alpha^2 ) \alpha \delta_{1}^2 \leq \delta_{2}, 
    \end{aligned} \]
    where the last inequality is from \eqref{eq:delta-1}, \eqref{eq:delta-2}, and $\alpha \leq 1 /M_0$. This implies that $\bx_{k+1} \in \mathcal{N}_{2}$. Hence, $\bx_{k+1} \in \mathcal{N}_R$. 
\end{proof}

Compared to the result in \cite{chen2021local}, the neighborhood $\mathcal{N}_R$ in Lemma \ref{lem:stay-neighborhood-rsi} is much more complicated due to the absence of the convex-like property of a general projection, i.e.,
\be \label{eq:proj-cvx} \| R_{x}(d) - y \| \leq \|x+ d - y\|, \ee
where $x, y \in \Mcal$ and $d\in T_x{\Mcal}$. Such inequality has been shown to hold for $\Mcal = {\rm St}(d,r)$ when $R$ is the projection operator (i.e., polar decomposition).  Another difference from \cite{chen2021local} is that the definition of
 $\delta_1$ in $\mathcal{N}_R$ depends on the number of agents, $N$. This dependence arises from the weaker results between consensus and gradients presented in Lemma \ref{lem:grad-consensus-sum2} and Lemma \ref{lem:grad-consensus-sum}. 

By applying the restricted secant inequality to the neighborhood $\mathcal{N}_R$ (which implies \eqref{eq:NR-rsi}),  we can establish the following result on linear convergence.
\begin{theorem}
Let $\{ \bx_{k} \}$ be the sequence generated by the Riemannian gradient descent and $\nu \in (0,1)$. If $t \geq \log_{\sigma_2} \frac{1}{2\sqrt{N}}$, $\bx_0 \in \mathcal{N}_R$, and $\alpha$ satisfies
$$\alpha \leq \min \left\{ \frac{\nu \Phi \tilde{\gamma}_R (1-2 \tilde{\gamma}_R)}{L_t(1-\tilde{\gamma}_R)((1-2\tilde{\gamma}_R)M_0^2L_t^2N\delta_{1}^2 + \tilde{\gamma}_R)} ,1, \frac{1}{M_0} \right \},$$ then $\{\bx_k\}$ converges linearly to the optimal solution set of problem \eqref{prob:original}, 
\be \label{eq:linear-rsi} \|\bx_{k+1} - \bar{\bx}_{k+1}\|^2 \leq (1-\alpha \tilde{\gamma}_R) \| \bx_k - \bar{\bx}_k \|^2 \leq (1-\alpha \tilde{\gamma}_R)^{k+1} \| \bx_0 - \bar{\bx}_0 \|^2. \ee
\end{theorem}
\begin{proof}
By Lemma \ref{lem:stay-neighborhood-rsi}, all iterates stay in the neighborhood $\mathcal{N}_R$. Then, using \eqref{eq:contrac-xk} and noting $\beta = (1 - 2\alpha \tilde{\gamma}_R)/ (\alpha\tilde{\gamma}_R)$, we have
\[ \begin{aligned}
    \| \bx_{k+1} - \bar{\bx}_{k+1}\|^2 & \leq (1 + \frac{1}{\beta})(1 - 2\alpha\tilde{\gamma}_R) \|\bx_k - \bar{\bx}_k\|^2 \\
    & \leq (1-\alpha\tilde{\gamma}_R) \| \bx_k - \bar{\bx}_k\|^2.
\end{aligned}
  \]
  This completes the proof.
\end{proof}
\begin{remark}
    It follows that as $\bx_k \rightarrow \mathcal{X}$, we have  $\Phi \rightarrow 2$. If we set $\nu=1/2$  choose an admissible step size  $\alpha = 1$, the convergence rate 
    \[ \sqrt{1 - \alpha \tilde{\gamma}_R} \rightarrow \sqrt{1- \frac{\mu_t}{8}}, \]
    which is worse than the rate $\sigma_2^t$ in the Euclidean setting and the rate $\sqrt{1-\mu_t}$ in the case of $\Mcal$ being the Stiefel manifold. This is due to the weak bound between the gradients and the consensus error, as well as the absence of  the convex-like inequality \eqref{eq:proj-cvx}. 
\end{remark}

\section{Connection between local Lipschitz continuity and RSI}
In Sections \ref{sec:Lip} and \ref{sec:rsi}, we have demonstrated the local linear convergence of the Riemannian gradient descent by establishing the conditions of local Lipschitz continuity of $P_{\Mcal}$ and the restricted secant inequality, respectively. It is natural to inquire about any potential connections or implications between these two conditions with respect to problem \eqref{prob:original}. In this section, we aim to address this question.

We will begin by introducing the local error bound condition \cite{luo1993error}, which is commonly employed in the analysis of local linear convergence for gradient descent or proximal gradient descent methods \cite{karimi2016linear}.
First, we present a lemma that establishes the equivalence between two types of local error bound conditions. These conditions can be viewed as generalizations of the Luo-Tseng error bound \cite{luo1993error} for problems in the Euclidean space. 
\begin{lemma} \label{lem:eb}
    Denote $\mathcal{N}(\delta) =\{\bx: \|\bx - \bar{\bx}\| \leq \delta \}$. If there exists some positive constants $\delta$ and $c$ such that 
    \be \label{eq:eb1} \|\grad \varphi^t(\bx) \| \geq c \|\bx - \bar{\bx} \|, \quad \bx \in \mathcal{N}(\delta), \ee
    then it holds that for some positive constants $\hat{\delta}$ and $\hat{c}$,
    \be \label{eq:eb2} \|\bx - P_{\Mcal^N}(\bW^t \bx) \| \geq \hat{c} \|\bx - \bar{\bx} \|, \quad \bx \in \mathcal{N}(\hat{\delta}). \ee 
    In addition, the converse is also true.
\end{lemma}
\begin{proof}
    It follows from \cite[Lemma 4.3]{chen2021local} that there exists a constant $Q > 0$ such that for all $\bx$,
    \[ \|\bx - P_{\Mcal^N}(\bW^t \bx) - \grad \varphi^t(\bx)\| \leq Q \|\bx - \bar{\bx} \|^2. \]
    This gives
    \[ \| \bx - P_{\Mcal^N}(\bW^t \bx) \| \geq \| \grad \varphi^t(\bx) \| - Q \|\bx - \bar{\bx} \|^2 \geq (c - Q \|\bx - \bar{\bx} \|) \|\bx - \bar{\bx} \|.  \]
    Then, taking $\hat{\delta} = \min\{ c/(2Q), \delta \}$, the inequality \eqref{eq:eb2} holds with $\hat{c} = c/2$. 
Conversely, if \eqref{eq:eb2} holds, then \eqref{eq:eb1} holds with $c=\hat{c}/2$ and $\delta = \min\{ \hat{c}/(2Q), \hat{\delta} \}$. 
\end{proof}

Next, we show the local Lipschitz continuity of $P_{\Mcal}$ as shown in \eqref{eq:lip-proj-alpha}  implies the satisfaction of the local error bound conditions  \eqref{eq:eb1} and \eqref{eq:eb2}. 
\begin{theorem} \label{lem:eb2}
    Let $\beta\in (0,2)$ be given. For the consensus problem \eqref{prob:original} with $t > \log_{\sigma_2} \frac{2-\beta}{2}$, the local error bound condition \eqref{eq:eb2} holds with $\hat{c} = (1 - 2 \sigma_2^t/(2-\beta))$ and $\hat{\delta} = \beta\gamma$.
\end{theorem}
\begin{proof}
    By \eqref{eq:lip-proj-alpha}, for any $\bx \in \mathcal{N}(\beta\gamma)$, 
    \[ \begin{aligned}
    \| \bx - P_{\Mcal^N}(\bW^t \bx) \| & \geq \| \bx - \bar{\bx}\| - \| \bar{\bx} -  P_{\Mcal^N}(\bW^t \bx) \|  \\
    & \geq \| \bx - \bar{\bx}\| - \frac{2}{2-\beta} \|\hat{\bx} - \bW^t \bx \| \\
    & \geq (1 - \frac{2\sigma_2^t}{2-\beta}) \| \bx - \bar{\bx} \|. 
    \end{aligned} \]
    By noting $t > \log_{\sigma_2} \frac{2-\beta}{2}$, we have that $\hat{c} > 0$. This completes the proof. 
\end{proof}

Furthermore, we can also derive a restricted secant inequality using the local Lipschitz continuity of $P_{\Mcal}$.
\begin{theorem} \label{thm:Lip-rsi}
    Let $\beta\in (0,2)$ be given. For the consensus problem \eqref{prob:original} with $t > \log_{\sigma_2} \frac{2-\beta}{2}$, the following restrict secant inequality holds,
    \be \label{eq:rsi-new} \iprod{\bx - \bar{\bx}}{\grad \varphi^t(\bx)} \geq c_r \| \bx - \bar{\bx} \|^2 \geq c_r/2 \| \grad \varphi^t(\bx) \|^2, \quad \forall x \in \mathcal{N}(\delta_r),\ee
    where $c_r = (4-(2-\beta)^2\sigma_2^{2t})/16 $ and $\delta_r= \min\{\beta \gamma, \left(4 - (2-\beta)^2\sigma_2^{2t} \right)/ (16Q) \}$. 
\end{theorem}
\begin{proof}
     Note that there exists a $Q > 0$ such that for any $\bx \in \mathcal{N}(\beta\gamma)$,  
     \[ \begin{aligned}
        & \iprod{\bx - \bar{\bx}}{\grad \varphi^t(\bx)} \\
        = & \iprod{\bx - \bar{\bx}}{\bx - P_{\Mcal}(\bW^t \bx) +  \grad \varphi^t(\bx) - \bx + P_{\Mcal}(\bW^t \bx)}   \\
        \geq & \iprod{\bx - \bar{\bx}}{\bx - P_{\Mcal}(\bW^t \bx)} - Q\| \nabla \varphi^t(\bx) \|^2 \|\bx - \bar{\bx}\| \\
        \geq & \iprod{\bx - \bar{\bx}}{\bx - \bar{\bx} + \bar{\bx} - P_{\Mcal}(\bW^t \bx)} -  4 Q \| \bx - \bar{\bx} \|^3 \\
        \geq & \| \bx - \bar{\bx}\|^2 - \frac{\|\bx - \bar{\bx} \|^2 + \| \bar{\bx} - P_{\Mcal}(W^t \bx) \|^2}{2} - 4 Q \| \bx - \bar{\bx} \|^3 \\
        \geq &  \left(\frac{1}{2} - \frac{(2-\beta)^2}{8}\sigma_2^{2t} \right) \| \bx - \bar{\bx} \|^2 - 4Q \| \bx - \bar{\bx}\|^3 \\
        \geq & \left(\frac{1}{4} - \frac{(2-\beta)^2}{16}\sigma_2^{2t} \right) \| \bx - \bar{\bx}  \|^2,
     \end{aligned} \]
     where the first inequality is from the Lipschitz-type inequality of $P_\Mcal$ \cite[Lemma 4.3]{deng2023decentralized}, the second inequality is due to the 2-Lipschitz continuity of $\nabla \varphi^t$, the third inequality comes from the basic inequality $\iprod{a}{b} \geq -(\|a\|^2 + \|b\|^2)/2$, the fourth inequality is from the 2-Lipschitz continuity of $P_{\Mcal}$, and we use $\|\bx - \bar{\bx}\| \leq \left(4 - (2-\beta)^2\sigma_2^{2t} \right)/ (16Q)$ in the last inequality. In addition, note that 
     \[ \|\grad \varphi^t(\bx)\| \leq \| \nabla \varphi^t(\bx) \| \leq 2 \| \bx - \bar{\bx} \|,  \]
     the restrict secant inequality \eqref{eq:rsi-new} holds. 
\end{proof}

We note that the above theorem relies on the condition $t > \log_{\sigma_2} (2-\beta)/2 > 1$, which is not necessary for Theorem \ref{thm:rsi}. This implies that the local Lipschitz continuity of $P_{\Mcal}$ can yield the restricted secant inequality if $t$ is sufficiently large. Additionally, based on Lemma \ref{lem:eb} and Theorem \ref{lem:eb2}, the local Lipschitz continuity also implies the local error bounds \eqref{eq:eb1} and \eqref{eq:eb2}. 
Furthermore, disregarding the discrepancy in the definition of the local neighborhood,  the restricted secant inequality \ref{eq:rsi-new} serves as a sufficient condition for the local error bound. 

\section{Numerical experiments}
In this section, we evaluate the performance of the Riemannian gradient descent (RGD) and the projected gradient descent (PGD) for solving the consensus problem \eqref{prob:original}.  For each algorithm, we randomly generate the initial value $\bx_0 \in \Mcal^N$, and we set the number of agents to be $N = 15$. For the comparisons, we report the consensus error and the norm  of the Riemannian gradient of $\varphi^t(\bx)$ for each algorithm, i.e., $\frac{1}{N}\|\bm{x} - \bar{\bm{x}}\|$ and $\|\grad \varphi^t(\bx)\|$. All algorithms are terminated either when the consensus error is less than or equal to $2\times 10^{-16}$, or when the number of iterations exceeds 1000.
\subsection{Consensus problem on the Stiefel manifold}
In this subsection, we evaluate the PGD and the RGD for solving the consensus problem \eqref{prob:original} on the Stiefel manifold, i.e., $\Mcal = \mathrm{St}(d,r):=\{x\in\mathbb{R}^{d\times r}:\;x^Tx = I_r\}$, where $d = 200$ and $r=2$.   For the RGD, we use the QR decomposition as the retraction. Additionally, we also test the RGD with the projection on ${\rm St}(d,r)$, i.e., the polar decomposition, as the retraction, which we refer to as PRGD.   

We evaluate the performance of different algorithms using three types of graphs, namely ``random'', ``star'', and ``cycle''. These graphs are generated according to \cite{shi2015extra}. In our experiments, we set  $\alpha = 1$ and $t=1$. The results, as depicted in Figure \ref{fig:diff-method}, include the consensus error and the norm of the Riemannian gradient under the different graph types. These results indicate that RGD and PRGD outperform PGD on different graphs. This can be attributed to the utilization of Riemannian gradients in RGD and PRGD, which enables them to leverage the underlying manifold structure. The overlapping trajectories of RGD and PRGD indicate that the choice of different contraction operators has a negligible impact on the performance of RGD.
Furthermore, it can be observed from the number of iterations that the performance obtained on the ``random'' graph is superior to those obtained on other graphs. This is attributed to the denser adjacency matrix $W$ in random graphs, albeit at the expense of increased communication costs. 
\begin{figure*}[!htb]
	\begin{center}
		\begin{minipage}[b]{0.32\linewidth}
			\centering
\centerline{\includegraphics[width=\linewidth]{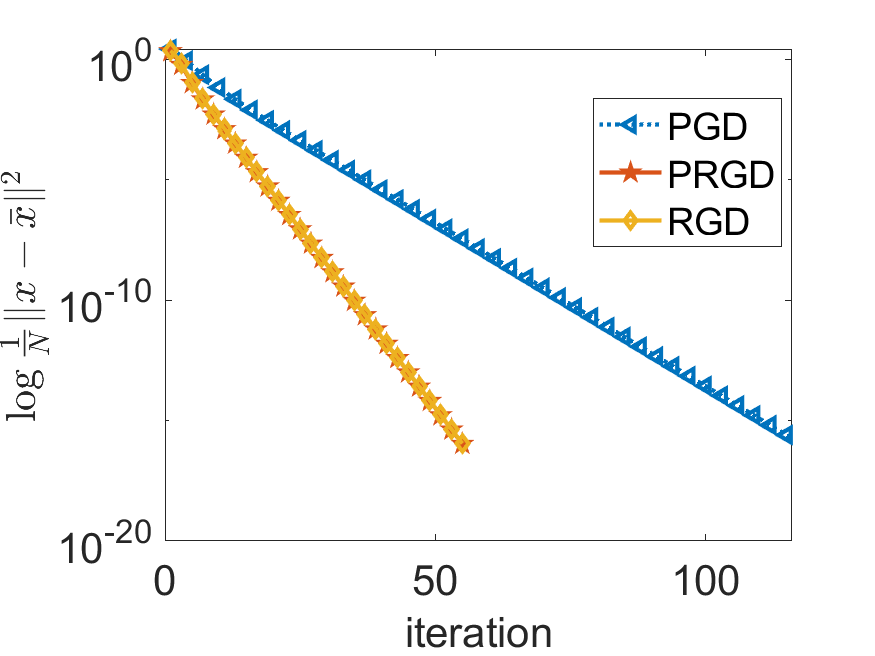}}
			\centerline{}\medskip
		\end{minipage}
		\begin{minipage}[b]{0.32\linewidth}
			\centering
			\centerline{\includegraphics[width=\linewidth]{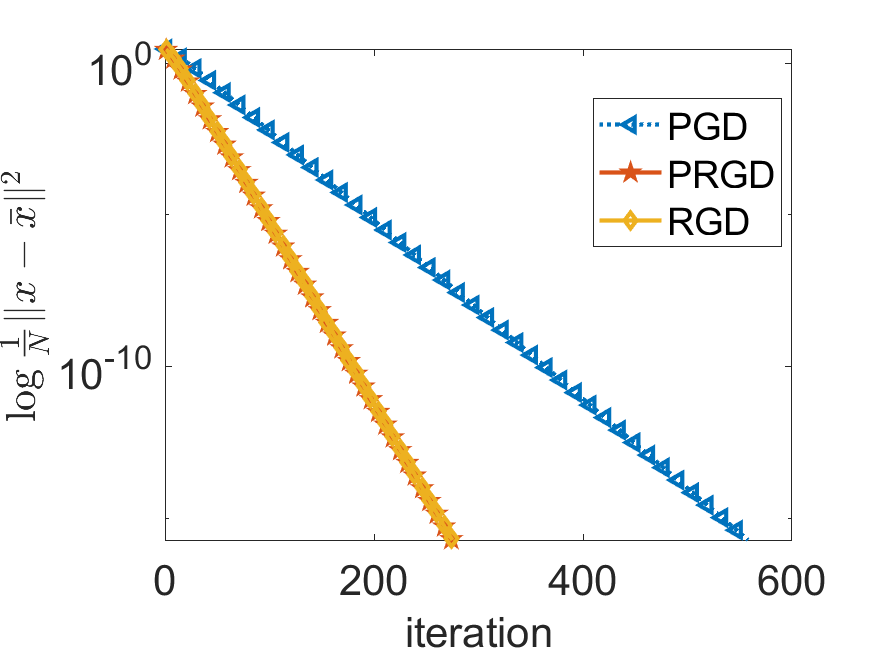}}
			\centerline{} \medskip
		\end{minipage}
  \begin{minipage}[b]{0.32\linewidth}
			\centering
			\centerline{\includegraphics[width=\linewidth]{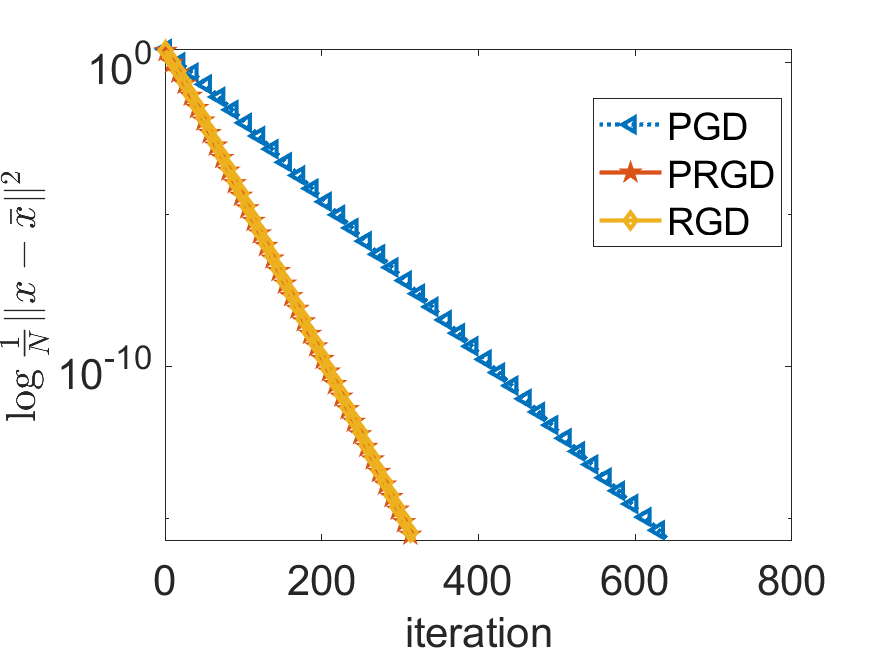}}
			\centerline{}\medskip
		\end{minipage}
  \\
  \begin{minipage}[b]{0.32\linewidth}
			\centering
			\centerline{\includegraphics[width=\linewidth]{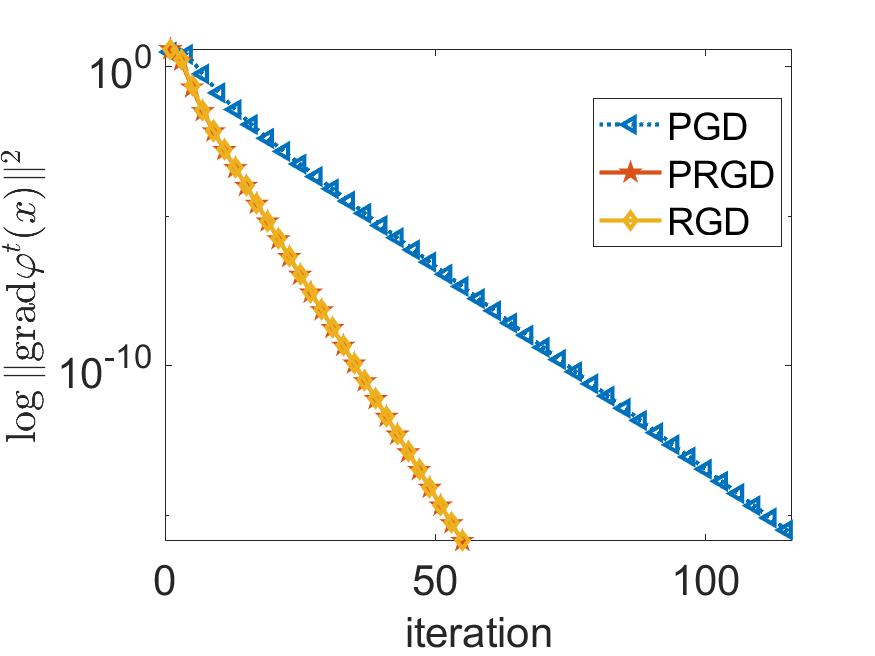}}
			\centerline{}\medskip
		\end{minipage}
    \begin{minipage}[b]{0.32\linewidth}
			\centering
			\centerline{\includegraphics[width=\linewidth]{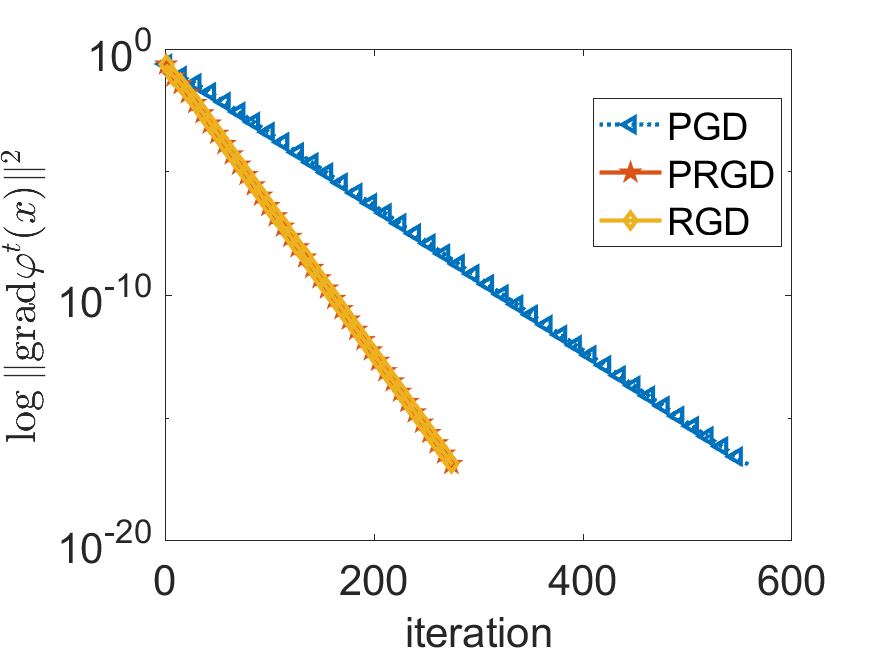}}
			\centerline{}\medskip
		\end{minipage}
    \begin{minipage}[b]{0.32\linewidth}
			\centering
			\centerline{\includegraphics[width=\linewidth]{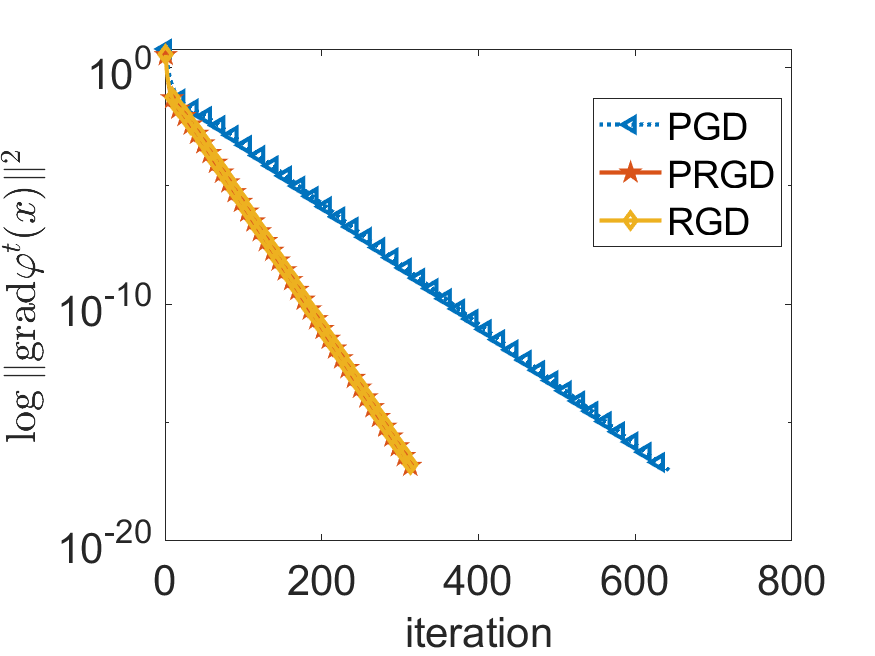}}
			\centerline{}\medskip
		\end{minipage}
	\end{center}
	\vskip -0.2in
	\caption{ Numerical results for solving consensus over the Stiefel manifold on different graphs with $\alpha = 1, t = 1$. Left:random; middle:star; right:cycle.}
	\label{fig:diff-method}
\end{figure*}

We also investigate the impact of different step sizes $\alpha$ and $t$ on the performance of the algorithms. Considering the similarity in performance among RGD with different retractions, we  test  PGD and  RGD with QR decomposition as the retraction. We consider two choices of step size: $\alpha=1,2/(L+\mu)$, where $L = 1-\lambda_{\min}(W)$ and $\mu$ denote the second-largest eigenvalue. Moreover, we also consider two choices of $t$: $t = 1, t=10$. Figures \ref{fig:diff-alpha-star} and \ref{fig:diff-alpha-cycle} depict the performance of these algorithms on the "cycle" and "star" graphs, respectively. For both figures, (a)(b) show the comparisons of different $\alpha$ with $t = 1$, (c)(d) show the comparisons of different $t$ with $\alpha = 1$.  Our observations indicate that the step size 
$\alpha = 2/(L+\mu)$ achieves better performance for both algorithms and graphs. It is worth emphasizing that obtaining the constants $L$ and $\mu$ is often challenging, especially in scenarios where the communication graph is dynamic. Setting $\alpha = 1$ is more practical if it is acceptable. Moreover, it is observed that both algorithms exhibit faster convergence rates when $t=10$ compared to when $t=1$, albeit at the cost of increased communication cost.

\begin{figure}[!htb]
    \centering
    \subfigure[]{
        \includegraphics[width=0.45\textwidth]{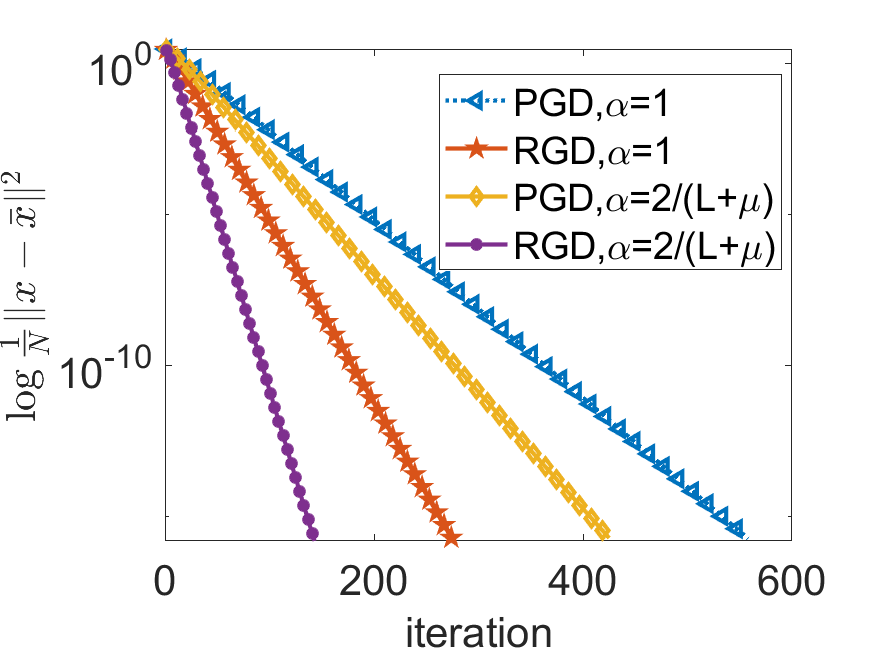}
    }
    \subfigure[]{
        \includegraphics[width=0.45\textwidth]{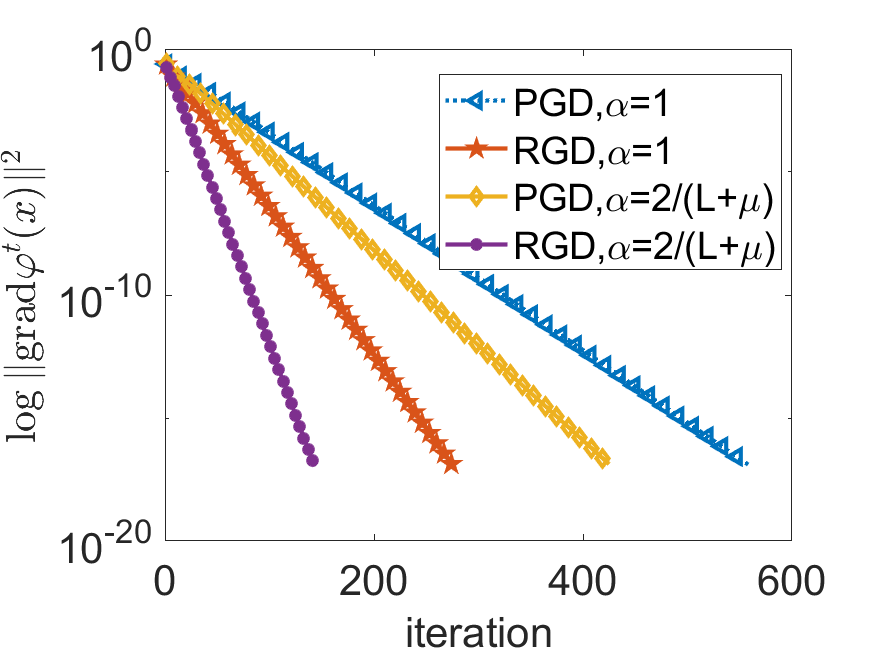}
    }
    \\
    
    \subfigure[]{
        \includegraphics[width=0.45\textwidth]{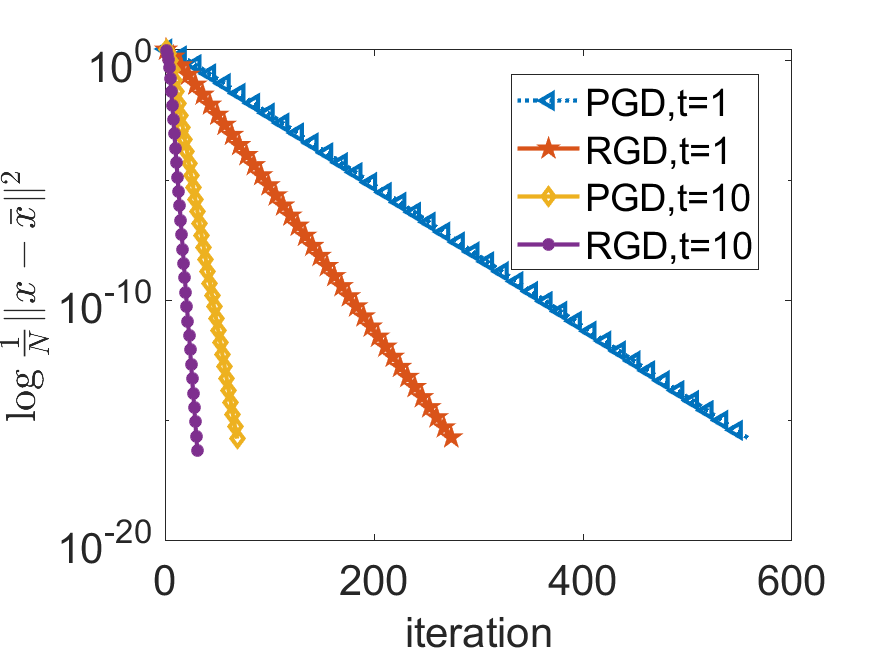}
    }
    \subfigure[]{
        \includegraphics[width=0.45\textwidth]{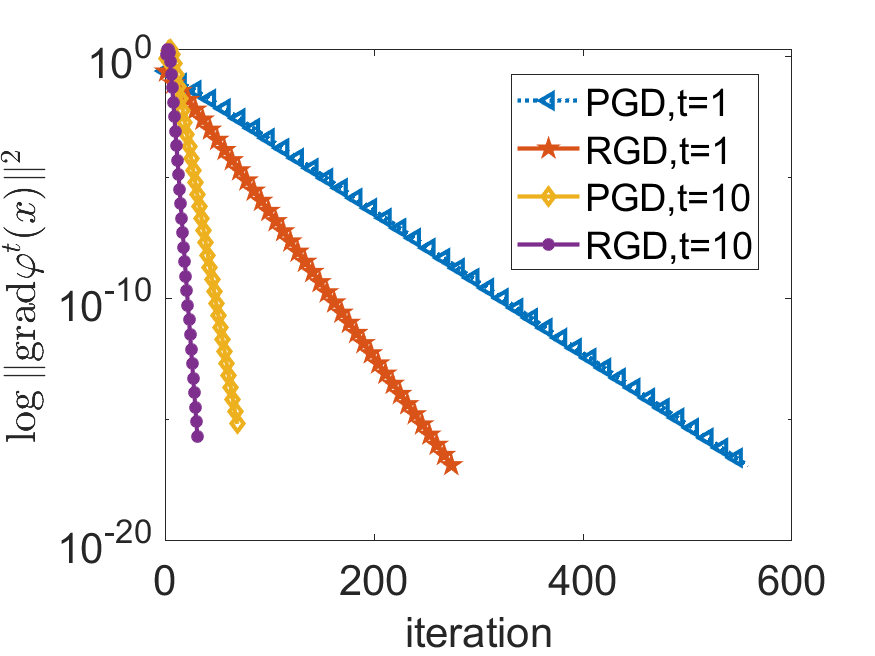}
    }
    \caption{Numerical results for solving consensus over the Stiefel manifold on ``star'' graph with different $\alpha$ and $t$. Top: fix $t=1$; botton: fix $\alpha = 1$.}
    \label{fig:diff-alpha-star}
\end{figure}

\begin{figure}[!htb]
    \centering
    \subfigure[]{
        \includegraphics[width=0.45\textwidth]{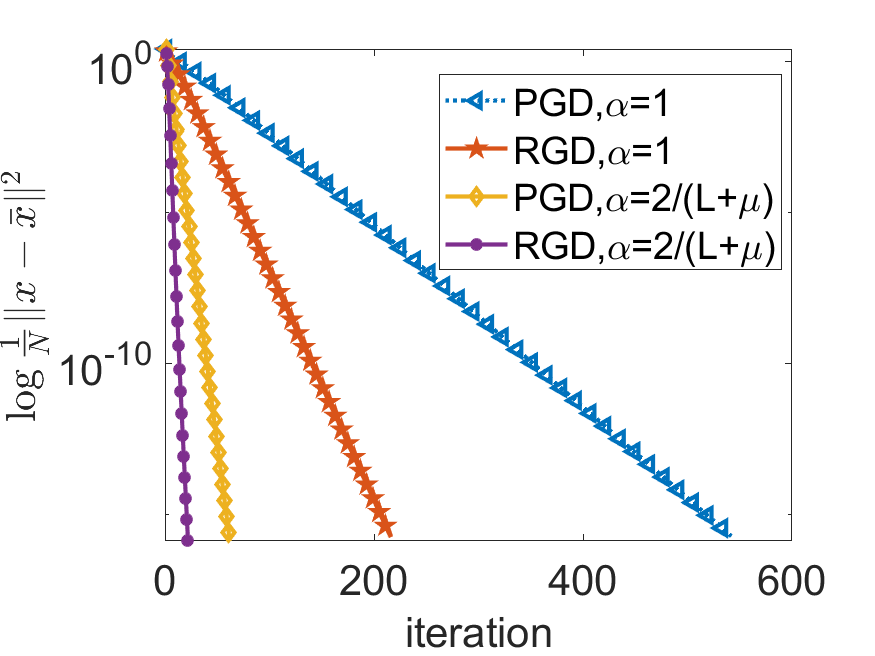}
    }
    \subfigure[]{
        \includegraphics[width=0.45\textwidth]{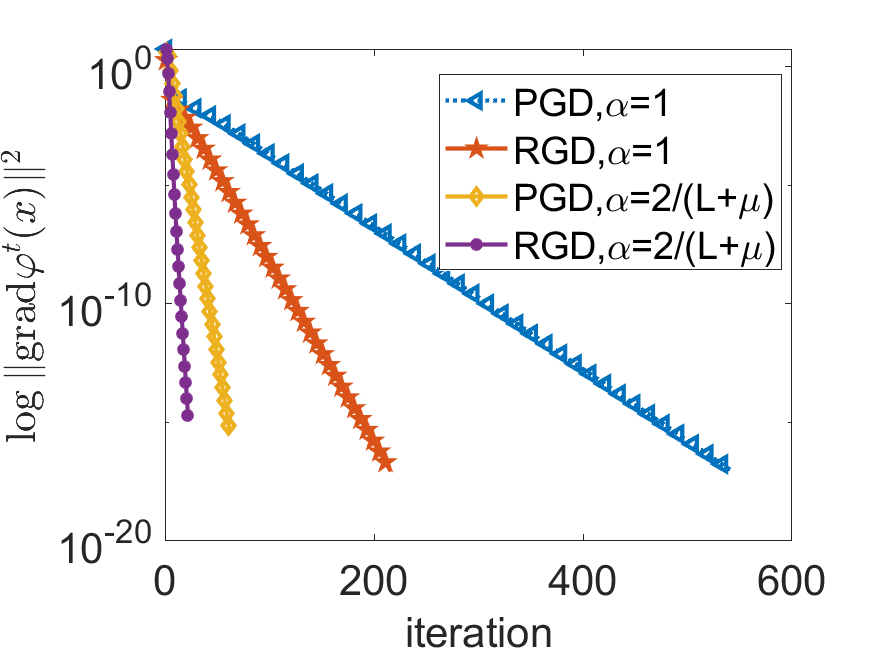}
    } \\
    
    \subfigure[]{
        \includegraphics[width=0.45\textwidth]{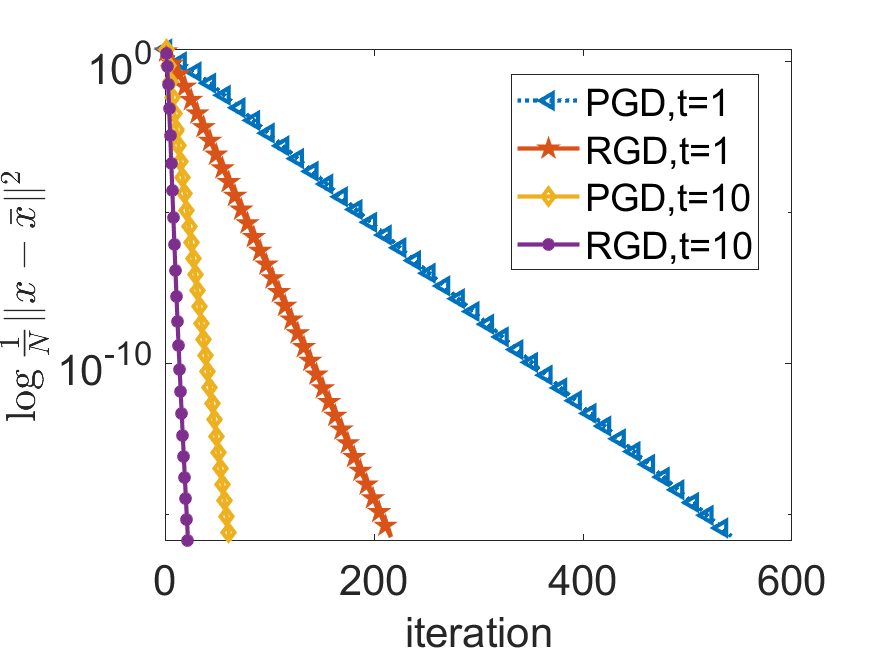}
    }
    \subfigure[]{
        \includegraphics[width=0.45\textwidth]{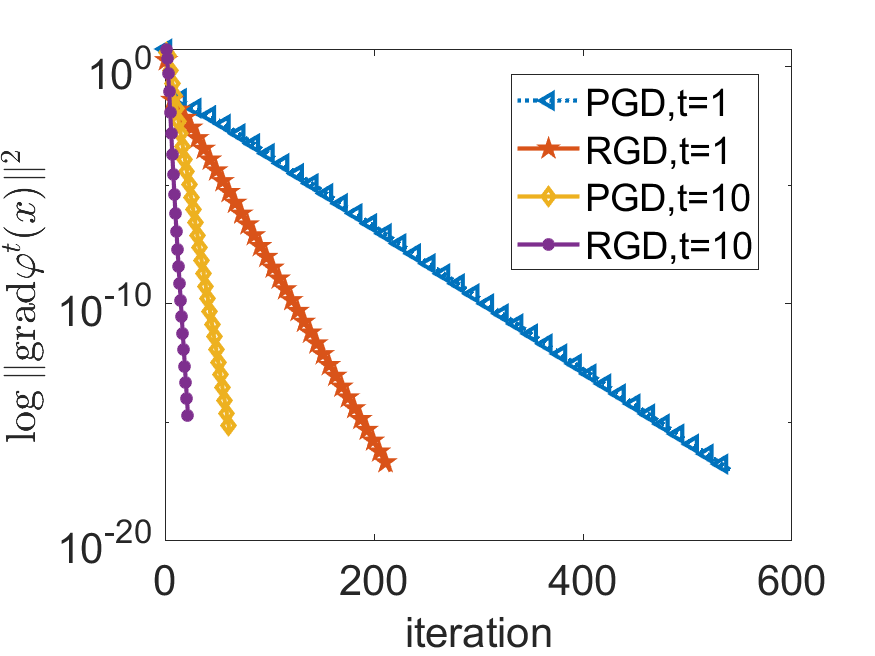}
    }
    \caption{Numerical results for solving consensus over the Stiefel manifold on ``cycle'' graph with different $\alpha$ and $t$. Top: fix $t=1$; bottom: fix $\alpha = 1$.}
    \label{fig:diff-alpha-cycle}
\end{figure}

\subsection{Consensus problem on the Oblique manifold}
In this subsection, we evaluate the performance of different algorithms for solving the consensus problem on the Oblique manifold, i.e., $\Mcal = \mathrm{Ob}(d,r):=\{x\in\mathbb{R}^{d\times r}:\mathrm{diag}(x^\top x)=\mathbf{1}_r\;\}$, where $d = 200$ and $r=5$, $\mathrm{diag}(X)$ is a vector whose $i$-th element is $X_{ii}$. Figures \ref{fig:diff-alpha-star-o} and \ref{fig:diff-alpha-cycle-o} show the performance of PGD and RGD on two graphs: ``star'', and ``cycle''.  The  results show that employing a step size of $\alpha = 2/(L+\mu)$ results in superior performance for both algorithms and graphs. In addition, both figures show that utilizing $t=10$ for the algorithms results in a faster convergence rate when compared to the cases where $t=1$. These findings are consistent with the results presented in the previous subsection.
\begin{figure}[!htb]
    \centering
    \subfigure[]{
        \includegraphics[width=0.45\textwidth]{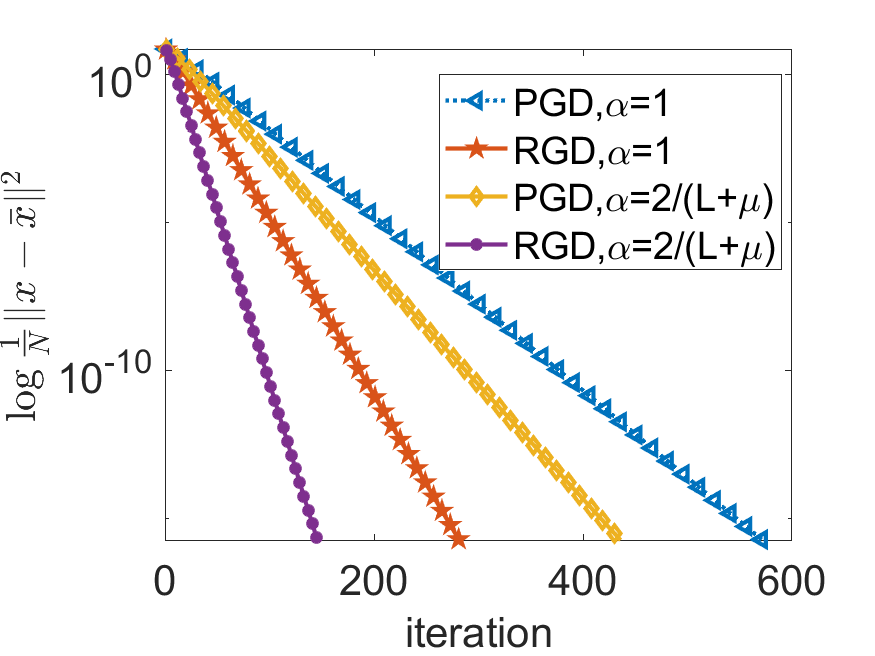}
    }
    \subfigure[]{
        \includegraphics[width=0.45\textwidth]{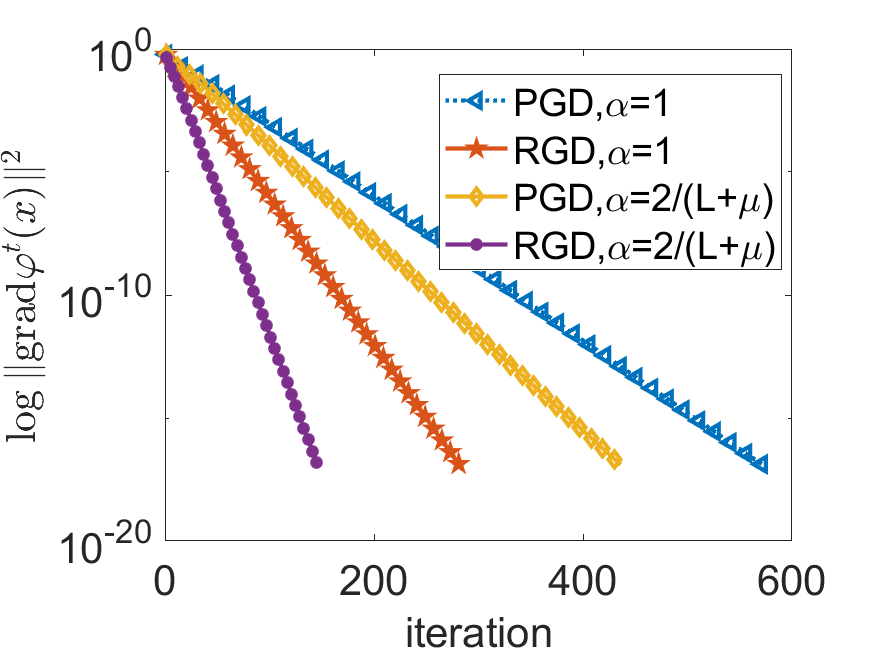}
    }\\
    \subfigure[]{
        \includegraphics[width=0.45\textwidth]{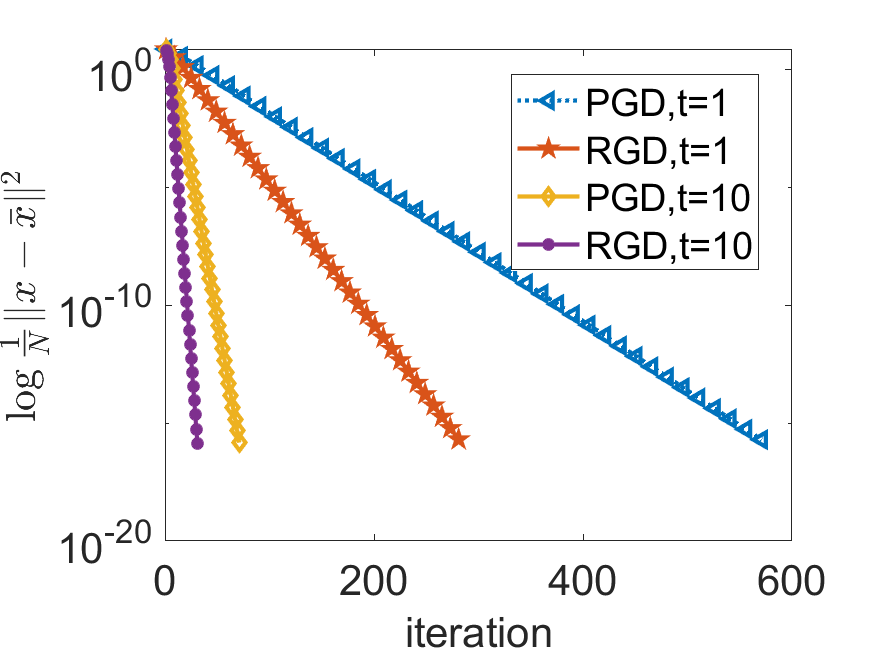}
    }
    \subfigure[]{
        \includegraphics[width=0.45\textwidth]{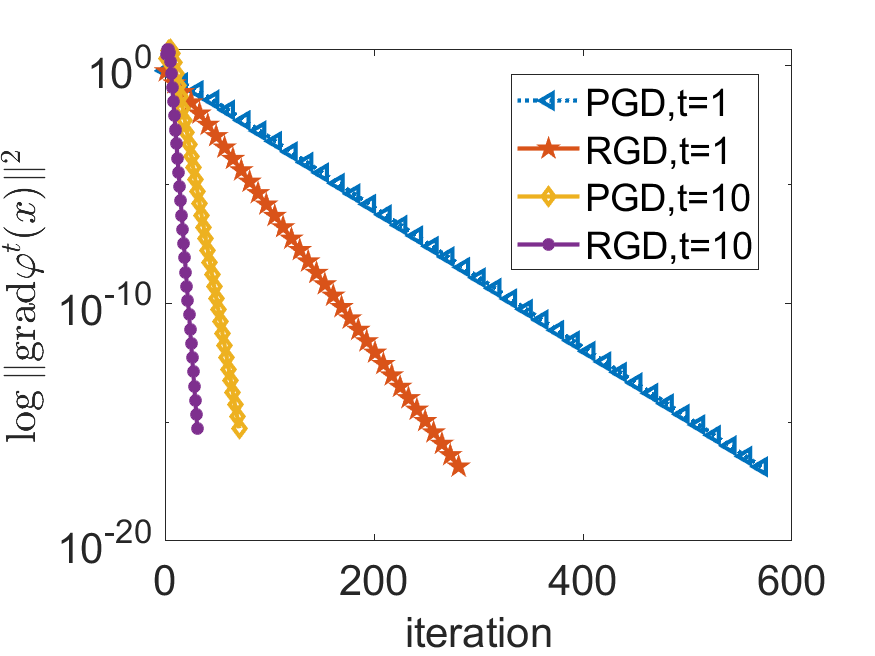}
    }
    \caption{Numerical results on ``star'' graph with different $\alpha$ and $t$. Top: fix $t=1$; bottom: fix $\alpha = 1$.}
    \label{fig:diff-alpha-star-o}
\end{figure}

\begin{figure}[!htb]
    \centering
    \subfigure[]{
        \includegraphics[width=0.45\textwidth]{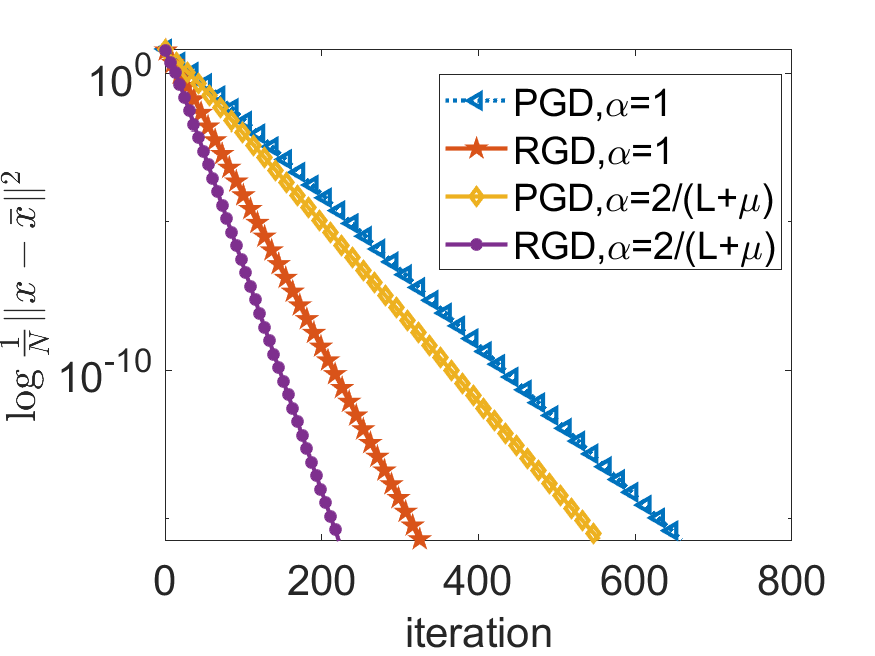}
    }
    \subfigure[]{
        \includegraphics[width=0.45\textwidth]{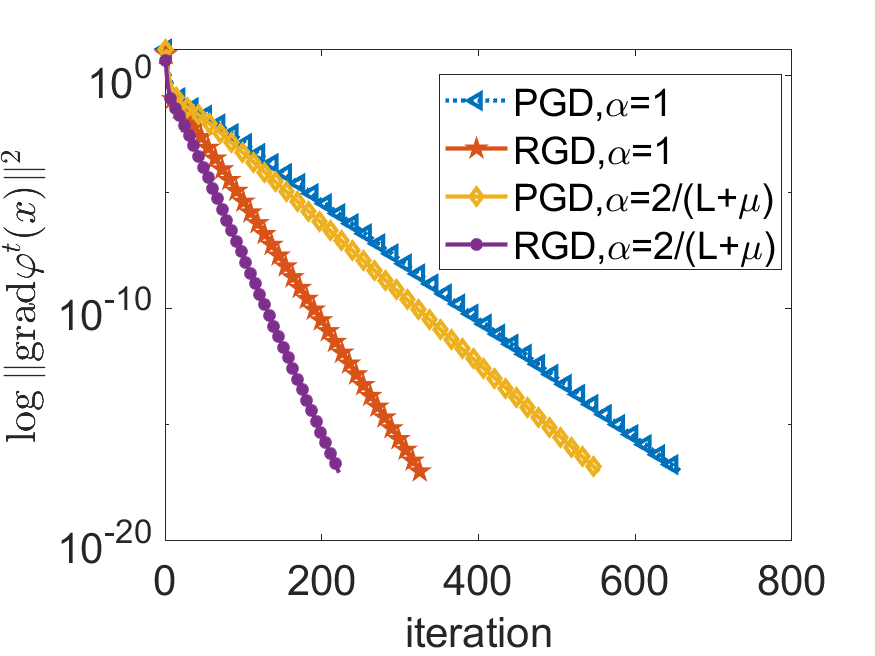}
    }\\
    \subfigure[]{
        \includegraphics[width=0.45\textwidth]{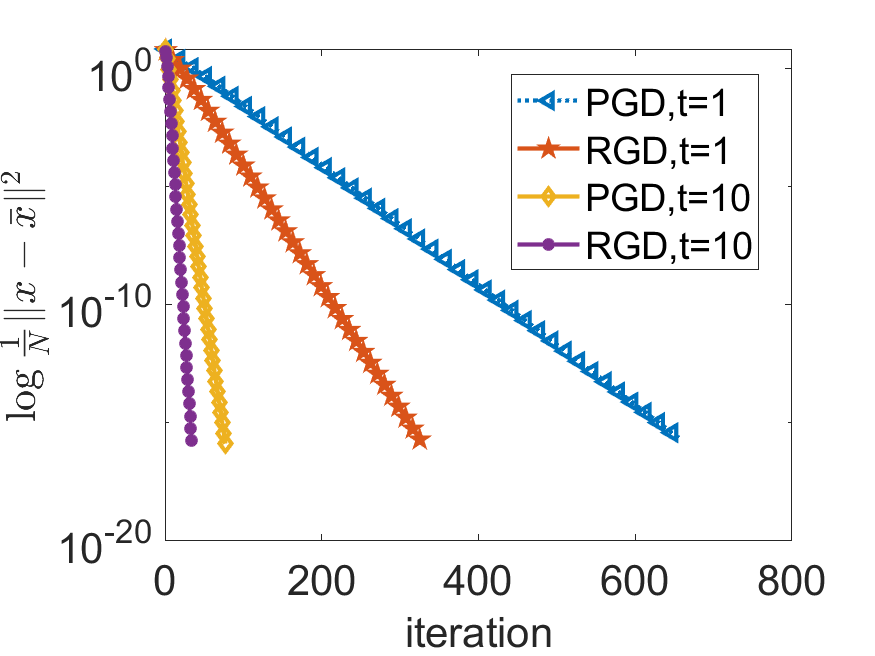}
    }
    \subfigure[]{
        \includegraphics[width=0.45\textwidth]{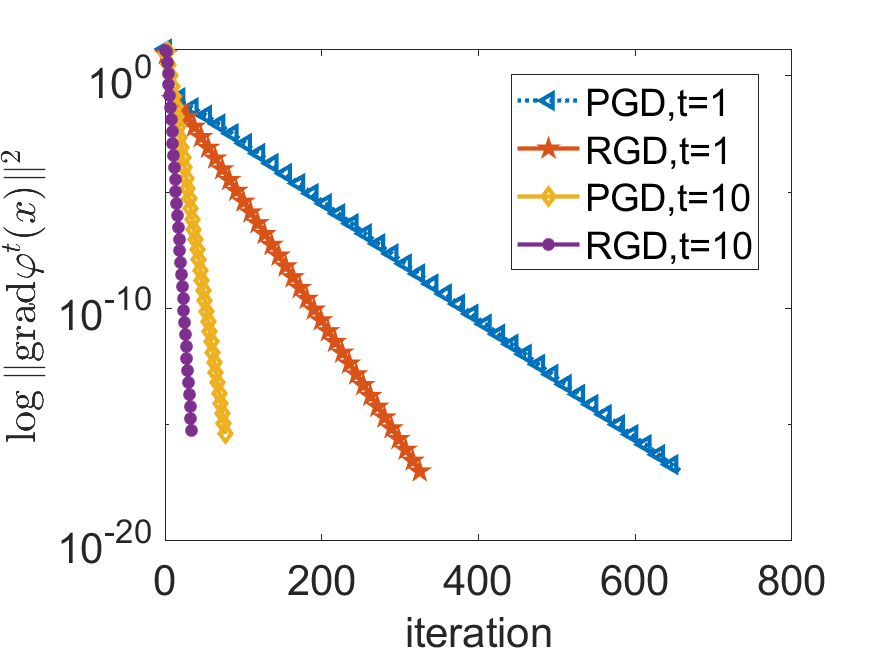}
    }
    \caption{Numerical results for solving consensus over the oblique manifold on ``cycle'' graph with different $\alpha$ and $t$. Top: fix $t=1$; bottom: fix $\alpha = 1$.}
    \label{fig:diff-alpha-cycle-o}
\end{figure}

\section{Conclusion}
In this paper, we present a demonstration of the effectiveness of Riemannian gradient descent in solving the consensus problem on a compact submanifold. By leveraging the geometric characteristics of the submanifold, we establish the presence of generalized convexity properties near the global optima, including the local Lipschitz continuity, the restricted secant inequality, and the local error bound, resulting in the linear convergence of the Riemannian gradient descent. The key tools  are the geometric properties of the tangent space and retraction operators, as well as the proximal smoothness of the compact submanifold. To validate our theoretical findings, we conduct numerical experiments that provide empirical evidence of the efficacy of our approach.

\bibliographystyle{siamplain}
\bibliography{ref}
\end{document}